\documentclass{amsart}
\usepackage{amssymb,amscd}
\advance\hoffset by -1in
\advance\voffset by -1in

\usepackage[cp1251]{inputenc}
\usepackage[russian]{babel}
\newtheorem{theorem}{Теорема}
\newtheorem{lemma}{Лемма}
\newtheorem{corollary}{Следствие}
\newtheorem{proposition}{Предложение}

\title{Теоремы о свободе для относительно свободных алгебр Ли и групп с одним определяющим соотношением}

\author{А.\,Ф.\,Красников}
\address{Омский государственный университет, г.\,Омск}
\email{phomsk@mail.ru}

\begin{document}

\maketitle

\section*{Введение}
\noindent Шуман \cite{Sch} доказал, что если
$X$ --- свободная группа, $R$ --- нормальная подгруппа в $X$, $v\in X$, $\mathfrak{X}$ --- фундаментальный идеал кольца ${\bf Z}(X)$, $\mathfrak{R}$ --- идеал, порожденный в ${\bf Z}(X)$ элементами
$\{r-1| r\in R\}$, то $v-1\in \mathfrak{R}\mathfrak{X}$ тогда и только тогда, когда $v\in [R,R]$.\\
Другими словами, элемент группы $X$ принадлежит $[R,R]$ тогда и только тогда, когда его производные Фокса \cite{Fx} равны нулю по модулю $R$.\\
Известная теорема о свободе Магнуса \cite{Mg} показывает, что если $R(x_1,\ldots,x_n)$ --- циклически несократимое слово в образующих
$x_1,\ldots,x_n$, содержащее $x_n$, то в группе $G=\langle x_1,\ldots,x_n;\text{ } R(x_1,\ldots,x_n)\rangle$ элементы
$x_1,\ldots,x_{n-1}$ являются свободными образующими порожденной ими подгруппы.\\
Теорема о свободе для групп с одним определяющим соотношением в многообразиях разрешимых и многообразиях нильпотентных групп данных ступеней доказана Романовским \cite{Rm1}, в многообразиях $\mathfrak{N}_c\mathfrak{A}$ --- Ябанжи \cite{Ya} (здесь $\mathfrak{A}$ --- многообразие абелевых групп, $\mathfrak{N}_c$ --- многообразие нильпотентных групп ступени, не превосходящей $c$), в полинильпотентных многообразиях --- Колмаковым \cite{Km}.\\
В настоящей работе получено обобщение результата Шумана:\\
{\sl Пусть $F$ --- свободная группа с базой $\{g_j | j\in J\}$, $K\subseteq J$, $F_K$
 --- подгруппа в $F$, порожденная $\{g_j| j\in K\}$; $v\in F$,
$N$--- нормальная подгруппа в $F$; $D_k~(k\in J)$ --- производные Фокса кольца
${\bf Z}(F)$. Тогда
\begin{eqnarray*}
D_k(v)\equiv ~0\mod{N},~k\in J\setminus K,
\end{eqnarray*}
если и только если найдется элемент $\hat{v}\in F_K$ такой, что $v\hat{v}^{-1}\in (F_K\cap N)^F[N,N]$.}
Затем доказывается теорема:\\
{\sl Пусть $F$ --- свободная группа с базой $y_1,\ldots,y_n$ $(n>2)$, $N$ --- вербальная подгруппа группы $F$, $F/N$ --- упорядочиваемая группа,
\begin{eqnarray}\label{vved_gr}
N=N_{11} \geqslant \ldots \geqslant N_{1,m_1+1}=N_{21} \geqslant \ldots \geqslant N_{s,m_s+1},
\end{eqnarray}
где $N_{ij}$ --- j-й член нижнего центрального ряда группы $N_{i1}$, $s \geqslant 1$.\\
Пусть, далее, $r\in N_{1,i}\backslash N_{1,i+1}\,(i\leqslant m_1)$, $R$ --- нормальная подгруппа, порожденная в группе $F$ элементом $r$, $H = \textup{гр }(y_1,\ldots,y_{n-1})$.\\
Если (и только если) элемент $r$ не сопряжен по модулю $N_{1,i+1}$ ни с каким словом от $y_1^{\pm 1},\ldots,y_{n-1}^{\pm 1}$, то
$H\cap RN_{kl} = H\cap N_{kl}$, где $N_{kl}$ --- произвольный член ряда {\rm (\ref{vved_gr})}}.\\
Так как свободные полинильпотентные группы упорядочиваемы \cite{Shm}, то из теоремы следует
результат Колмакова:\\
{\sl Пусть $F$ --- свободная группа с базой $y_1,\ldots,y_n$ $(n>2)$,
\begin{eqnarray}\label{vved2_gr}
F=F_{11} \geqslant \ldots \geqslant F_{1,m_1+1}=F_{21} \geqslant \ldots \geqslant F_{s,m_s+1},
\end{eqnarray}
где $F_{ij}$ --- j-й член нижнего центрального ряда группы $F_{i1}$, $s \geqslant 1$.\\
Пусть, далее, $r\in F_{i,j}\backslash F_{i,j+1}\,(j\leqslant m_i)$, $R$ --- нормальная подгруппа, порожденная в группе $F$ элементом $r$, $H = \textup{гр }(y_1,\ldots,y_{n-1})$.\\
Если (и только если) элемент $r$ не сопряжен по модулю $F_{i,j+1}$ ни с каким словом от $y_1^{\pm 1},\ldots,y_{n-1}^{\pm 1}$, то
$H\cap RF_{kl} = H\cap F_{kl}$, где $F_{kl}$ --- произвольный член ряда {\rm (\ref{vved2_gr})}}.\\
В работе Умирбаева \cite{Um} определяются частные производные элементов свободной алгебры Ли $L$ со значениями в универсальной обертывающей алгебре $U(L)$, которые являются аналогами производных Фокса в теории групп. Будем называть эти производные
производными Фокса алгебры $L$. Харлампович  \cite{Hm} доказала для алгебр Ли аналог теоремы Шумана:\\
{\sl Пусть $F$ --- свободная алгебра Ли с базой $\{g_j | j\in J\}$, $N$ --- идеал в $F$, $N_U$ --- идеал, порожденный $N$ в универсальной обертывающей алгебре $U(F)$, $v\in F$,
$D_j~(j\in J)$ --- производные Фокса алгебры $F$. Тогда
\begin{eqnarray*}
D_j(v)\equiv ~0\mod{N_U},~j\in J,
\end{eqnarray*}
если и только если $v\in [N,N]$.}\\
Известная теорема о свободе Ширшова \cite{Sh} показывает, что если $F$ --- свободная алгебра Ли с множеством $R$ образующих
и одним определяющим соотношением $s=0$, в левую часть которого входит образующий $x$, то подалгебра, порожденная в алгебре $F$ множеством $R\setminus x$, свободна.\\
Теорема о свободе для полинильпотентных алгебр Ли с одним определяющим соотношением доказана Талаповым \cite{Tl}.\\
В настоящей работе получено обобщение результата Харлампович:\\
{\sl Пусть $F$ --- свободная алгебра Ли с базой $\{g_j | j\in J\}$, $N$ --- идеал в $F$, $N_U$ --- идеал, порожденный $N$ в универсальной обертывающей алгебре $U(F)$, $v\in F$,
$D_j~(j\in J)$ --- производные Фокса алгебры $F$, $K\subseteq J$, $F_K$
 --- подалгебра в $F$, порожденная $\{g_j| j\in K\}$. Тогда
\begin{eqnarray*}
D_k(v)\equiv ~0\mod{N_U},~k\in J\setminus K,
\end{eqnarray*}
если и только если найдутся элементы $v_0\in F_K$ и $v_1$ из идеала, порожденного в $F$ подалгеброй $F_K\cap N$, такие, что $v\equiv ~v_0+v_1\mod{[N,N]}$.}\\
Затем доказывается теорема:\\
{\sl Пусть $F$ --- свободная алгебра Ли с базой $y_1,\ldots,y_n$ $(n>2)$,
$H$ --- подалгебра, порожденная в алгебре $F$ элементами $y_1,\ldots,y_{n-1}$, $N$ --- эндоморфно допустимый идеал алгебры $F$,
\begin{eqnarray}\label{vved}
N=N_{11} \geqslant \ldots \geqslant N_{1,m_1+1}=N_{21} \geqslant \ldots \geqslant N_{s,m_s+1},
\end{eqnarray}
где $N_{kl}$ --- l-я степень алгебры $N_{k1}$.\\
Пусть, далее, $r\in N_{1i}\backslash N_{1,i+1}\,(i\leqslant m_1)$, $R = \mbox{\rm ид}_F (r)$.\\
Если (и только если) $r\not\in H+N_{1,i+1}$, то $H\cap (R+N_{kl})=H\cap N_{kl}$, где $N_{kl}$ --- произвольный член ряда {\rm (\ref{vved})}}.\\
Из теоремы следует результат Талапова:\\
{\sl Пусть $F$ --- свободная алгебра Ли с базой $y_1,\ldots,y_n$ $(n>2)$,
$H$ --- подалгебра, порожденная в алгебре $F$ элементами $y_1,\ldots,y_{n-1}$,
\begin{eqnarray}\label{vved2}
F=F_{11} \geqslant \ldots \geqslant F_{1,m_1+1}=F_{21} \geqslant \ldots \geqslant F_{s,m_s+1},
\end{eqnarray}
где $F_{kl}$ --- l-я степень алгебры $F_{k1}$.\\
Пусть, далее, $r\in F_{ij}\backslash F_{i,j+1}\,(j\leqslant m_i)$, $R = \mbox{\rm ид}_F (r)$.\\
Если (и только если) $r\not\in H+F_{i,j+1}$, то $H\cap (R+F_{kl})=H\cap F_{kl}$, где $F_{kl}$ --- произвольный член ряда {\rm (\ref{vved2})}}.

\section{Некоторые свойства производных Фокса для групп}
\noindent Пусть $F=(\underset{i\in I}\ast A_i)\ast G$ --- свободное
произведение нетривиальных групп $A_i~(i\in I)$ и свободной группы
$G$ с базой $\{g_j | j\in J\}$, $N$--- нормальная подгруппа в $F$.\\
Обозначим через ${\bf Z}(F)$ целочисленное групповое кольцо группы
$F$. Дифференцированием кольца ${\bf Z}(F)$ называется отображение
$\partial:{\bf Z}(F)\to {\bf Z}(F)$ удовлетворяющее условиям
\begin{gather}
\partial(u+v)=\partial(u)+\partial(v),\notag\\
\partial(uv)=\partial(u)v+ \varepsilon (u)\partial(v)\notag
\end{gather}
для любых $u,~v\in {\bf Z}(F)$, где $\varepsilon$ --- гомоморфизм
тривиализации $F\to 1$, продолженный по линейности на ${\bf Z}(F)$.\\
Следуя Романовскому \cite{Rm}, обозначим через $D_k~(k\in I\cup J)$ производные Фокса кольца
${\bf Z}(F)$ --- дифференцирования, однозначно определяемые
условиями:
\begin{gather}
D_j(g_j)=1~(j\in J),~D_k(g_j)=0,~\mbox{при}~k\neq j;\notag\\
\mbox{если}~a_i\in A_i~(i\in
I),~\mbox{то}~D_i(a_i)=a_i-1,~D_k(a_i)=0,~\mbox{при}~k\neq
i.\notag
\end{gather}
Для $u\in {\bf Z}(F),~f\in F,~n\in N$ имеют место формулы:
\begin{gather}
D_k(f^{-1})=-D_k(f)f^{-1},~D_k(f^{-1}nf)\equiv D_k(n)f\mod{N};\notag\\
u-\varepsilon (u)=\sum_{i\in I} D_i(u)+\sum_{j\in J} (g_j-1)D_j(u).\label{1tm01_gr}
\end{gather}
Пусть $G$ --- группа; $A$, $B$ --- подмножества множества элементов группы $G$; $C$, $D$ --- подмножества множества элементов кольца ${\bf Z}(G)$. Обозначим через $\textup{гр}(A)$ подгруппу, порожденную $A$ в $G$, через $A^G$ --- нормальную подгруппу,
порожденную $A$ в $G$, через $\gamma_k G$ --- $k$-й член нижнего центрального ряда группы $G$. Если $x,~y$ --- элементы $G$, то положим $[x,y]=x^{-1}y^{-1}xy$, $x^y=y^{-1}xy$.
Через $AB$ обозначим множество произведений вида $ab$, где $a,~b$ пробегают соответственно элементы $A$, $B$, через $[A,B]$ --- подгруппу группы $G$, порожденную всеми $[a,b]$, $a\in A$, $b\in B$. Через $CD$ обозначим множество сумм произведений вида $cd$, где $c,~d$ пробегают соответственно элементы $C$, $D$.
\begin{theorem}\label{tm1_gr}
Пусть $F$ --- свободное произведение нетривиальных групп
$A_i~(i\in I)$ и свободной группы $G$ с базой $\{g_j | j\in J\}$, $K\subseteq I\cup J$, $F_K$
 --- подгруппа в $F$, порожденная $\{g_j| j\in K\cap J\}$ и
$\{A_i| i\in K\cap I\}$; $v\in F$,
$N$--- нормальная подгруппа в $F$, $M = \textup{гр
}((N\cap A_i)^{F}|i\in I)[N,N]$. Тогда
\begin{eqnarray}\label{1tm02_gr}
D_k(v)\equiv ~0\mod{N},~k\in (I\cup J)\setminus K,
\end{eqnarray}
если и только если найдется элемент $\hat{v}\in F_K$ такой, что $v\hat{v}^{-1}\in (F_K\cap N)^FM$.
\end{theorem}
\begin{proof}
Непосредственно проверяется, что из $v\hat{v}^{-1}\in (F_K\cap N)^FM$, $\hat{v}\in F_K$, следуют сравнения (\ref{1tm02_gr}).
Необходимо доказать обратное.\\
Формулы (\ref{1tm01_gr}), (\ref{1tm02_gr}) показывают, что
\begin{eqnarray}\label{1tm01-02_gr}
{v-1}\equiv     \sum_{i\in K\cap I} D_i(v) + \sum_{j\in K\cap
J}(g_j-1)D_j(v)\mod{N}.
\end{eqnarray}
Обозначим через $\Delta_{F_K}$ фундаментальный идеал кольца ${\bf Z}(F_K)$.
Из (\ref{1tm01-02_gr}) следует, что при естественном гомоморфизме
${\bf Z}(F)\to {\bf Z}(F/N)$ образ элемента $v-1$ принадлежит
образу $\Delta_{F_K}{\bf Z}(F)$. Хорошо известно, что тогда $v \in F_KN$
(доказательство см., например, в \cite{Re}),
т.е. найдется $\hat{v}\in F_K$ такой, что $v\hat{v}^{-1}\in N$. Элемент $v\hat{v}^{-1}$ обозначим через $w$.
Предположим, что $w\notin (F_K\cap N)^FM$ и приведем это предположение к противоречию.
Пусть $X$ --- множество элементов из $F$, полученное объединением $\{g_j | j\in J\}$ с множеством элементов групп $A_i~(i\in I)$; $u\to \bar u$ --- функция, выбирающая правые шрайеровы представители $F$ по $N$,
$S$ --- множество выбранных представителей. Тогда $N = \text{гр }(sx{\overline{sx}}^{-1}|s\in S,~x\in X)$ (доказательство
см., например, в \cite{KrMr}). Следовательно, найдутся $s_ix_i{\overline{s_ix_i}}^{-1},~s_i\in S,~x_i\in
X$ и ненулевые целые $k_i$ такие, что
\begin{eqnarray*}
w\equiv (s_1x_1{\overline{s_1x_1}}^{-1})^{k_1}\ldots (s_lx_l{\overline{s_lx_l}}^{-1})^{k_l}\mod{(F_K\cap N)^FM}
\end{eqnarray*}
и элемент $w$ нельзя представить по модулю $(F_K\cap N)^FM$ в виде произведения степеней меньшего чем $l$ числа элементов вида
$sx{\overline{sx}}^{-1},~s\in S,~x\in X$. Через $w_i$ будем обозначать элементы $s_ix_i{\overline{s_ix_i}}^{-1}$,
$i=1,\ldots,l$.
Отметим, что
\begin{eqnarray}\label{1tm03_gr}
\sum_{i=1}^l k_i D_q(w_i)\equiv 0 ~mod~N,~q\in (I\cup J)\setminus K.
\end{eqnarray}
Предположим, что в $\{x_1,\ldots,x_l\}$ есть не принадлежащий $F_K$ элемент. Пусть это будет $x_1$.
Приведем это предположение к противоречию.
Выберем $q\in (I\cup J)\setminus K$ такое, что либо $x_1=g_q$ либо $x_1\in A_q$.
Рассмотрим случай $x_1=g_q$. Если $x_i=g_q$, то $D_q(w_i)= D_q(s_i){s_i}^{-1}w_i+
{\overline{s_ix_i}}^{-1}-D_q(\overline{s_ix_i}){\overline{s_ix_i}}^{-1}$. Следовательно,
\begin{eqnarray}\label{1tm04_gr}
D_q(w_i)\equiv D_q(s_i){s_i}^{-1}+
{\overline{s_ix_i}}^{-1}-D_q(\overline{s_ix_i}){\overline{s_ix_i}}^{-1}~mod~N.
\end{eqnarray}
Если $x_i\neq g_q$, то $D_q(w_i)= D_q(s_i){s_i}^{-1}w_i-
D_q(\overline{s_ix_i}){\overline{s_ix_i}}^{-1}$. Следовательно,
\begin{eqnarray}\label{1tm05_gr}
D_q(w_i)\equiv
D_q(s_i){s_i}^{-1}-D_q(\overline{s_ix_i}){\overline{s_ix_i}}^{-1}~mod~N.
\end{eqnarray}
Пусть $t\in S$ и $t=u{g_q}^\varepsilon u_1$, где $\varepsilon = \pm 1$ и слово $u_1$ не
содержит в своей записи буквы $g_q$. Тогда
\begin{eqnarray*}
D_q(t){t}^{-1}=
D_q(u){u}^{-1}+D_q({g_q}^\varepsilon)({u{g_q}^\varepsilon})^{-1},
\end{eqnarray*}
т.е. $D_q(t){t}^{-1}$ --- сумма элементов вида $D_q({g_q}^\varepsilon)({u{g_q}^\varepsilon})^{-1}$, $u{g_q}^\varepsilon\in S$.
Покажем, что
\begin{eqnarray*}
{\overline{s_1x_1}}^{-1}\not\equiv \pm
D_q({g_q}^\varepsilon)({u{g_q}^\varepsilon})^{-1}~mod~N.
\end{eqnarray*}
Предположим противное. При $\varepsilon = 1$ будем иметь ${\overline{s_1x_1}}^{-1}\equiv (u{g_q})^{-1}~mod~N$.
Но тогда $s_1=u$ и
$\overline{s_1x_1}=s_1x_1$. Пришли к
противоречию.\\
При $\varepsilon = -1$ будем иметь ${\overline{s_1x_1}}^{-1}\equiv {u}^{-1}~mod~N$.\\
Но тогда $s_1=ux_1^{-1}$ и
$s_1x_1{\overline{s_1x_1}}^{-1}=1$.
Пришли к противоречию.\\
Теперь можно утверждать, что из (\ref{1tm03_gr}), (\ref{1tm04_gr}),
(\ref{1tm05_gr}) следует существование $i$ такого, что $i\neq 1,~x_i=g_q$ и
${\overline{s_ix_i}}^{-1}={\overline{s_1x_1}}^{-1}$. Но
тогда $s_i=s_1$ и потому $w_i=w_1$. Снова пришли к противоречию.\\
Рассмотрим оставшийся случай $x_1\in A_q$. Обозначим через $L$ подмножество в
$\{1,\ldots,l\}$, состоящее из индексов тех элементов $x_i$, которые принадлежат
$A_q$. Выберем элементы  $u_i$, $v_i$, $a_i$, $i\in L$, так, чтобы было:
$u_i$, $v_i$ --- элементы из $S$ и не кончаются символом из $A_q$, $a_i\in A_q$,
$w_i=u_ia_i{v_i}^{-1}$, $i\in L$.\\
Если $i\in L$, то
\begin{eqnarray}\label{1tm06_gr}
D_q(w_i)\equiv D_q(u_i){u_i}^{-1}+
{u_i}^{-1}-{v_i}^{-1}-D_q(v_i){v_i}^{-1}~mod~N.
\end{eqnarray}
Если $i\notin L$, то
\begin{eqnarray}\label{1tm07_gr}
D_q(w_i)\equiv
D_q(s_i){s_i}^{-1}-D_q(\overline{s_ix_i}){\overline{s_ix_i}}^{-1}~mod~N.
\end{eqnarray}
Предположим, что $\sum_{j\in L} k_j({u_j}^{-1}-{v_j}^{-1})\neq 0$.
Пусть $y\in F$ и $y=y_1y_2$, где слово $y_2$ не содержит в своей записи
символов из $A_q$. Тогда $D_q(y){y}^{-1}=D_q(y_1){y_1}^{-1}$,
поэтому из (\ref{1tm03_gr}), (\ref{1tm06_gr}), (\ref{1tm07_gr}) следует, что найдутся кончающиеся символом из $A_q$ попарно
различные элементы $z_1,\ldots,z_n$ из $S$ и целые, не равные нулю,
числа $m_1,\ldots,m_n$ такие, что
\begin{eqnarray}\label{1tm08_gr}
\sum_{j\in L} k_j({u_j}^{-1}-{v_j}^{-1})+\sum_{i=1}^n m_i
D_q(z_i){z_i}^{-1}\equiv 0 ~mod~N.
\end{eqnarray}
Пусть $z_i=\hat{z}_i a_i$, где $a_i\in A_q$. Тогда
\begin{eqnarray}\label{1tm09_gr}
D_q(z_i){z_i}^{-1} = D_q(\hat{z}_i){\hat{z}_i}^{-1} +
{\hat{z}_i}^{-1} - {z_i}^{-1}.
\end{eqnarray}
Из (\ref{1tm08_gr}), (\ref{1tm09_gr}) вытекает $m_i=0$ для тех $i$,
для которых слова $z_i$ имеют в своей записи максимальное число
вхождений символов из $A_q$. Пришли к противоречию.\\
Нами доказано, что $\sum_{j\in L} k_j({u_j}^{-1}-{v_j}^{-1})= 0$.
Так как $w_j\notin M$ и $u_ja_j{v_j}^{-1}\in N$, то
${u_j}^{-1}\neq {v_j}^{-1}~(j\in L)$.  Поэтому для любого элемента
$b\in \{u_i,~v_i\}$ найдется $\{u_j,~v_j\}~(j\neq i,~i,j\in L)$
такое, что
$b\in \{u_j,~v_j\}$.\\
Следовательно, найдутся $\{u_{i_1},~v_{i_1}\},
\ldots,\{u_{i_k},~v_{i_k}\}~(i_1,\ldots,i_k$ - попарно различные
элементы из $L$), попарно различные элементы $b_1,\ldots,b_k$ из
$S$\\
и $\varepsilon_i=\pm 1~(i=1,\ldots,k)$ такие, что
\begin{gather}
\{u_{i_k},~v_{i_k}\}\cap\{u_{i_1},~v_{i_1}\}=b_1,~\{u_{i_{j-1}},~v_{i_{j-1}}\}\cap\{u_{i_j},~v_{i_j}\}=b_j,~j=2,\ldots,k,\notag\\
{w_{i_1}}^{\varepsilon_1}\ldots
{w_{i_k}}^{\varepsilon_k}=b_1{x_{i_1}}^{\varepsilon_1}\ldots
{x_{i_k}}^{\varepsilon_k}{b_1}^{-1}.\notag
\end{gather}
Из ${x_{i_1}}^{\varepsilon_1}\ldots {x_{i_k}}^{\varepsilon_k}\in N\cap A_q$
следует, что
${w_{i_1}}^{\varepsilon_1}\ldots {w_{i_k}}^{\varepsilon_k}\in M$ в противоречии с выбором $w_1,\ldots,w_l$.\\
Полученные противоречия показывают, что $\{x_1,\ldots,x_l\}\subset F_K$.\\
Выберем минимальное $n$ с таким свойством: найдутся $f_i\in F_K$, $0\neq m_i\in {\bf Z}$ и не кончающиеся символом из $F_K\cap X$ элементы $v_i$, $\hat{v}_i\in S$, $v_if_i\hat{v}_i^{-1}\in N$ $(i=1,\ldots,n)$ такие, что
\begin{eqnarray}\label{1tm10_gr}
w\equiv (v_1f_1\hat{v}_1^{-1})^{m_1}\ldots (v_nf_n\hat{v}_n^{-1})^{m_n}\mod{(F_K\cap N)^FM}.
\end{eqnarray}
Числом с таким свойством будет, например, $l$.
Без потери общности рассуждений мы можем и будем считать, что $v_1$ - элемент максимальной длины среди
элементов $v_1, \hat{v}_1,\ldots,v_n, \hat{v}_n$ и ему нет равных в этом множестве элементов.
Действительно, если $\hat{v}_1$ - элемент максимальной длины, то заменим $v_1f_1\hat{v}_1^{-1}$ на $\hat{v}_1f_1^{-1}v_1^{-1}$;
если $v_1=\hat{v}_1$, то $v_1f_1\hat{v}_1^{-1}\in (F_K\cap N)^FM$ --- в противоречии с минимальностью $n$; если
$v_1=v_i$, $1\neq i$, то заменим $v_if_i\hat{v}_i^{-1}$ на $v_1f_1\hat{v}_1^{-1}\hat{v}_1f_1^{-1}f_i\hat{v}_i^{-1}$; если
$v_1=\hat{v}_i$, $1\neq i$, то заменим $v_if_i\hat{v}_i^{-1}$ на $(v_1f_1\hat{v}_1^{-1}\hat{v}_1f_1^{-1}f_i^{-1}v_i^{-1})^{-1}$.\\
Пусть $q\in (I\cup J)\setminus K$ такое, что либо $v_1$ кончается одним из символов $g_q$, $g_q^{-1}$ либо $v_1$ кончается символом из $A_q$. Из (\ref{1tm10_gr}) получаем
\begin{eqnarray}\label{1tm11_gr}
D_q(w)\equiv \sum_{i=1}^n m_i(D_q(v_i)v_i^{-1} - D_q(\hat{v}_i)\hat{v}_i^{-1})\mod{N}.
\end{eqnarray}
Если $u\in S$, то нетрудно видеть, что $D_q(u)u^{-1}$ будет суммой элементов вида $\pm t^{-1}$, $t\in S$, $t$ --- начальный отрезок слова $u$. Поэтому из (\ref{1tm11_gr}) следует, что если $v_1$ кончается символом $g_q$, то
\begin{eqnarray*}
D_q(w)\equiv m_1v_1^{-1} + \mu\mod{N},
\end{eqnarray*}
если $v_1$ кончается символом из $A_q$, то
\begin{eqnarray*}
D_q(w)\equiv -m_1v_1^{-1} + \mu\mod{N},
\end{eqnarray*}
если $v_1=\tilde{v}_1g_q^{-1}$, то
\begin{eqnarray*}
D_q(w)\equiv -m_1\tilde{v}_1^{-1} + \mu\mod{N},
\end{eqnarray*}
где $\mu$ --- сумма элементов вида $\pm t^{-1}$, $t\in S$, $t\neq v_1$ и, в случае $v_1=\tilde{v}_1g_q^{-1}$, $t\neq \tilde{v}_1$.
Тогда $D_q(w)\not\equiv 0 ~mod~N$ --- в противоречии с (\ref{1tm03_gr}).
\end{proof}

\begin{corollary} \cite{Rm}
Пусть $F$ --- свободное произведение свободной группы $G$ с базой $\{g_j | j\in J\}$ и нетривиальных групп
$A_i~(i\in I)$,
$N$--- нормальная подгруппа в $F$ такая, что $N\cap A_i = 1~(i\in I)$. Тогда
\begin{eqnarray*}
D_k(v)\equiv ~0\mod{N},~k\in (I\cup J),
\end{eqnarray*}
если и только если $v\in [N,N]$.
\end{corollary}

\noindent Из доказательства теоремы~\ref{tm1_gr} вытекает
\begin{corollary}\label{lm4_2_gr}
Пусть $G$ --- свободная группа с базой с базой $\{g_j | j\in J\}$,
$H$ --- нормальная подгруппа группы $G$, $S$ --- шрайерова система представителей $G$ по $H$, $u\rightarrow \bar{u}$ --- соответствующая ей выбирающая
функция, $\{D_j | j\in J\}$ --- производные Фокса кольца ${\bf Z}(G)$.\\
Если $w_0=sg_{j_0}\overline{sg_{j_0}}^{\,-1}\neq 1$ ($s\in S)$, то
\begin{eqnarray*}
D_{j_0}(w_0) = \overline{sg_{j_0}}^{\,-1}+\sum_{p={1}}^{d} \delta_p t_p^{-1},
\end{eqnarray*}
где $\delta_p\in {\bf Z}(H)$, $t_p\in S$, $t_p\neq \overline{sg_{j_0}}$, $p\in \{1,\ldots,d\}$.\\
Если $w=ty\overline{ty}^{\,-1}$, $t\in S$, $y\in \{g_j | j\in J\}$, $ty\overline{ty}^{\,-1}\neq sg_{j_0}\overline{sg_{j_0}}^{\,-1}$, то
\begin{eqnarray*}
D_{j_0}(w) = \sum_{q={1}}^{l} \mu_q s_q^{-1},
\end{eqnarray*}
где $\mu_q\in {\bf Z}(H)$, $s_q\in S$, $s_q\neq \overline{sg_{j_0}}$, $q\in \{1,\ldots,l\}$.
\end{corollary}

\section{Теорема о свободе для относительно свободных групп с одним определяющим соотношением}
\noindent Пусть $F$ --- свободная группа с базой $y_1,\ldots,y_n$, $H$ --- подгруппа группы $F$.
Назовем смежный класс группы $F$ по подгруппе $H$ $\alpha$-классом, если в нем есть слово от $y_1^{\pm 1},\ldots,y_{n-1}^{\pm 1}$ и $\beta$-классом в противном случае. Назовем длиной $\alpha$-класса длину самого короткого слова от $y_1^{\pm 1},\ldots,y_{n-1}^{\pm 1}$ в нем. В $\alpha$-классах выберем представителей индукцией по длине класса.\\
Выберем пустое слово в качестве представителя для $H$. Если длина $\alpha$-класса равна 1, то выберем в нем любое
слово от $y_1^{\pm 1},\ldots,y_{n-1}^{\pm 1}$ длины 1 в качестве представителя этого класса. Пусть в $\alpha$-классах длины,
меньшей $l$, представители уже выбраны, т.е. на этих классах уже определена выбирающая
функция $u\rightarrow \bar{u}$. Пусть $L$ --- произвольный $\alpha$-класс длины $l$.
Возьмем в нем какое-нибудь слово $z_1\ldots z_l$, $z_m\in \{y_1^{\pm 1},\ldots,y_{n-1}^{\pm 1}\}$,
и объявим представителем класса $L$ слово $\overline{z_1\ldots z_{l-1}}z_l$.\\
Назовем длиной $\beta$-класса длину самого короткого слова в нем.
В $\beta$-классах выберем представителей индукцией по длине класса.
Если длина $\beta$-класса равна 1, то выберем в нем любое
слово длины 1 в качестве представителя этого класса. Пусть в $\beta$-классах длины,
меньшей $l$, представители уже выбраны, т.е. на этих классах и всех $\alpha$-классах уже определена выбирающая
функция $u\rightarrow \bar{u}$. Пусть $L$ --- произвольный $\beta$-класс длины $l$.
Возьмем в нем какое-нибудь слово $z_1\ldots z_l$, $z_m\in \{y_1^{\pm 1},\ldots,y_n^{\pm 1}\}$,
и объявим представителем класса $L$ слово $\overline{z_1\ldots z_{l-1}}z_l$.
Ясно, что так построенная система представителей, обозначим ее через $S$, шрайерова.
Будем обозначать подсистему в $S$, состоящую из представителей $\alpha$-классов через $S_\alpha$,
состоящую из представителей $\beta$-классов через $S_\beta$.
Элементы из $S_\alpha$ будем называть $\alpha$-представителями, из $S_\beta$ --- $\beta$-представителями.
Неединичные элементы вида $(Ka\overline{Ka}^{-1})^{\pm 1}$, $K\in S_\alpha$, $a\in \{y_1,\ldots,y_{n-1}\}$ будем называть $\alpha$-порождающими. Неединичные элементы вида $(Ka\overline{Ka}^{-1})^{\pm 1}$, $K\in S$, $a\in \{y_1,\ldots,y_n\}$, не являющиеся $\alpha$-порождающими, будем называть $\beta$-порождающими.\\
Нетрудно видеть, что любое слово от $y_1^{\pm 1},\ldots,y_{n-1}^{\pm 1}$ из $H$ можно записать в виде слова от $\alpha$-порождающих.\\
Пусть $F$ --- свободная группа, $H$ --- подгруппа группы $F$,
$N=N_1 \geqslant \ldots \geqslant N_t \geqslant \ldots $ --- ряд
нормальных подгрупп группы $F$ с абелевыми факторами без кручения, $[N_i,N_j\,]\leqslant N_{i+j}$, $G$ --- конечно-порожденная (к.п.) подгруппа группы $N/N_l$, $\varphi$ --- естественный гомоморфизм $F\rightarrow F/N_l$,
продолженный по линейности на ${\bf Z}(F)$.
Ряд
\begin{eqnarray*}
G=G_1\geqslant \ldots\geqslant G_l=1,~G_t=G\cap \varphi(N_t),
\end{eqnarray*}
центральный с факторами без кручения.
Полагаем $H_i=\varphi(H)\cap G_i$.
Мальцевскую базу $M=\{a_1,\ldots ,a_s\}$ группы $G$ будем выбирать так, чтобы $M=M_1\cup M_2$, $M_1\cap M_2=\emptyset$, элементы базы группы $H_1$ из $H_i\setminus H_{i+1}$ были степенями по модулю $G_{i+1}$ некоторых элементов из $M_1$
и если $m\in M_1$, $m\in G_i\setminus G_{i+1}$, то некоторая степень $m$ --- элемент базы группы $H_1$ из $H_i\setminus H_{i+1}$ по модулю $G_{i+1}$.\\
Ясно, что если $a_j\in M_2,~a_j\in G_{k(j)}\setminus G_{k(j)+1}$, то $(a_j)\cap H_1 G_{k(j)+1} = 1$.
Произведение $(a_1-1)^{\alpha_1}\cdots (a_s-1)^{\alpha_s}$, где $\alpha_j$ --- целые неотрицательные числа, называется мономом от $(a_1-1),\ldots,(a_s-1)$. Положим $\omega(a_j-1)=k\Leftrightarrow a_j\in G_k\setminus G_{k+1}$.
Сумма $\alpha_1\omega(a_1-1) + \ldots + \alpha_s\omega(a_s-1)$ называется весом монома. Мономы упорядочиваем по весу. Мономы равного веса упорядочиваем по длине. Полагая $(a_i-1)<(a_j-1)$, если $i<j$, мономы равного веса и равной длины упорядочиваем лексикографически (слева направо).\\
Пусть $\Delta_k$ --- идеал в ${\bf Z}(G)$, порожденный $(G_{i_1} - 1)\cdots (G_{i_t} - 1)$, $i_1 + \cdots + i_t\geqslant k$,
$\Delta_0={\bf Z}(G)$.
Тогда мономы веса $k$ образуют базу ${\bf Z}$-модуля $\Delta_k\mod {\Delta_{k+1}}$.
Если $\nu\in \Delta_k$, то через $\bar{\nu}$ будем обозначать линейную комбинацию мономов веса $k$ такую, что $\overline{\nu}\equiv \nu\mod{\Delta_{k+1}}$.\\
Непосредственно проверяется, что
\begin{gather}
n(a-1)\equiv (a^n-1)\mod {\Delta_{k+1}},~a \in G_k\setminus G_{k+1};\label{fm01}\\
(a-1)(b-1)=(b-1)(a-1)+ba([a,b]-1).\label{fm02}
\end{gather}
Если $m$ --- моном, то $H$-компонентой $m$ будем называть моном, получающийся вычеркиванием в $m$ тех
$a_j-1$, у которых $a_j\in M_2$.
Будем обозначать через $L(m)$ длину $m$, через $L_H(m)$ длину $H$-компоненты $m$.
Определим функцию $\psi$ на элементах вида
\begin{eqnarray*}
z_1m_1+\ldots+z_km_k,
\end{eqnarray*}
$z_1,\ldots,z_k$ --- целые числа, $m_1,\ldots,m_k$ --- мономы, полагая
\begin{eqnarray}\label{psi}
\psi(z_1m_1+\ldots+z_km_k)=\max_t\, (L(m_t)-L_H(m_t)).
\end{eqnarray}
Используя формулу (\ref{fm01}) нетрудно доказать, что произведение мономов веса $i$ и $j$, совпадающих со своими $H$-компонентами, равно по модулю $\Delta_{i+j+1}$ линейной комбинации мономов веса $i+j$, каждый из которых совпадает со своей $H$-компонентой.
Поэтому, если $\nu$, $\mu$  --- элементы из ${\bf Z}(G)$ то
формула (\ref{fm02}) показывает, что $\psi(\overline{\nu\mu})=\psi(\bar{\nu})+\psi(\bar{\mu})$.\\
Пусть $\hat{G}$ --- к.п. подгруппа группы $N/N_l$, $\hat{G}\geqslant G$.
По аналогии с $M$, $\Delta_k$ и $\psi$ определяем мальцевскую базу $\hat{M}=\{\hat{a}_1,\ldots ,\hat{a}_{\hat{s}}\}=\hat{M}_1\cup \hat{M}_2$ группы $\hat{G}$, идеал $\hat{\Delta}_k$ в ${\bf Z}(\hat{G})$ и функцию $\hat{\psi}$.
Если $\nu\in \hat{\Delta}_k$, то через $\hat{\nu}$ будем обозначать линейную комбинацию мономов веса $k$ такую, что $\hat{\nu}\equiv \nu\mod{\hat{\Delta}_{k+1}}$.\\
Если $a_j\in M_1$, то $\widehat{a_j-1}$ --- линейная комбинация мономов $\hat{a}_i-1$, $\hat{a}_i\in \hat{M}_1$, т.е. $\psi(\bar{\nu})\geqslant\hat{\psi}(\hat{\nu})$, $\nu\in {\bf Z}(G)$.
Следовательно для произвольного конечного множества $\nu_1,\ldots ,\nu_p$ элементов кольца ${\bf Z}(N/N_l)$ можно выбрать
такую к.п. подгруппу $G$ группы $N/N_l$, что $\{\nu_1,\ldots ,\nu_p\}\subset {\bf Z}(G)$ и, для любой к.п. подгруппы $\hat{G}\geqslant G$ группы $N/N_l$, $\psi(\bar{\nu_i})=\hat{\psi}(\hat{\nu_i})$ $(i=1,\ldots ,p)$.\\
Пусть $h\in \varphi(H)$; $\nu_1$, $\nu_2=h^{-1}\nu_1 h\in {\bf Z}(G)$ и, для любой к.п. подгруппы $\hat{G}\geqslant G$ группы $N/N_l$, $\psi(\bar{\nu_i})=\hat{\psi}(\hat{\nu_i})$ $(i=1,2)$.\\
Покажем, что $\psi(\bar{\nu_1})=\psi(\bar{\nu_2})$.
Пусть, для определенности, $\psi(\bar{\nu_1})\leqslant \psi(\bar{\nu_2})$.
Обозначим $\textup{гр }(a_1,\ldots ,a_s,\,a_1^h,\ldots ,a_s^h)$ через $\hat{G}$.
Ясно, что если $a_j\in M_1\cap \varphi(N_k)$, то $h^{-1} a_j h$ принадлежит по модулю $\varphi(N_{k+1})$ группе, порожденной
элементами из $\hat{M}_1\cap \varphi(N_k)$. Отсюда $\psi(\bar{\nu_1})\geqslant\hat{\psi}(\hat{\nu_2})=\psi(\bar{\nu_2})$, т.е. $\psi(\bar{\nu_1})=\psi(\bar{\nu_2})$.

\begin{lemma}\label{lm6_2_gr}
Пусть $F$ --- свободная группа с базой $y_1,\ldots,y_n$, $H = \textup{гр }(y_1,\ldots,y_{n-1})$,
$N=N_1 \geqslant \ldots \geqslant N_t \geqslant \ldots $ --- ряд
нормальных подгрупп группы $F$ с абелевыми факторами без кручения, $[N_i,N_j\,]\leqslant N_{i+j}$, $S_\alpha$ --- система $\alpha$-представителей группы $F$ по подгруппе $N$, $F/N$ --- упорядочиваемая группа, $\Delta_k$ --- идеал в ${\bf Z}(N)$, порожденный $(N_{i_1} - 1)\cdots (N_{i_s} - 1)$, $i_1 + \cdots + i_s \geqslant k$, ${\Delta}^\prime_k={\bf Z}(H)\cap\Delta_k$.
Пусть, далее, $v\in S_\alpha {\Delta}^\prime_{l-1}\setminus S_\alpha {\Delta}^\prime_l$;
$r \in S_\alpha \Delta_{j-1}$;
$w \in S_\alpha \Delta_{l-j}$.
Если
\begin{eqnarray*}
v \equiv rw\mod{{\bf Z}(F)\Delta_l},
\end{eqnarray*}
то найдется натуральное $C$ такое, что $Cr\in S_\alpha {\Delta}^\prime_{j-1}\mod{{\bf Z}(F)\Delta_j}$.
\end{lemma}
\begin{proof}
Обозначим через $\varphi$ естественный гомоморфизм $F\rightarrow F/N_l$, продолженный по линейности на ${\bf Z}(F)$.
Так как ${\bf Z}(F)(N_l - 1)\subseteq {\bf Z}(F)\Delta_l$, то лемма будет доказана, если мы покажем, что
найдется натуральное $C$ такое, что $C\varphi(r)\in \varphi(S_\alpha {\Delta}^\prime_{j-1})\mod{\varphi({\bf Z}(F)\Delta_j)}$.\\
Пусть $G$ --- к.п. подгруппа группы $N/N_l$,
\begin{eqnarray*}
G=G_1\geqslant \ldots\geqslant G_l=1,~G_t=G\cap \varphi(N_t).
\end{eqnarray*}
Полагаем $\overline{H}=\varphi(H)$, $\psi$ --- функция, определяемая формулой (\ref{psi}). 
Пусть $\bar{\Delta}_k$ --- идеал в ${\bf Z}(G)$, порожденный $(G_{i_1} - 1)\cdots (G_{i_t} - 1)$, $i_1 + \cdots + i_t\geqslant k$,
${\bf Z}(G)=\bar{\Delta}_0$.
Если $\nu\in \bar{\Delta}_k$, то через $\bar{\nu}$ будем обозначать линейную комбинацию мономов веса $k$ такую, что $\overline{\nu}\equiv \nu\mod{\bar{\Delta}_{k+1}}$.\\
Мы можем и будем считать, что
\begin{eqnarray}\label{lm6_2_0}
\varphi(v) \equiv \varphi(r)\varphi(w)\mod{\varphi(S_\alpha)\bar{\Delta}_l},
\end{eqnarray}
$\varphi(v) = \sum_i g_i \mu_i$, $\varphi(r) = \sum_p f_p \nu_p$, $\varphi(w) = \sum_k \lambda_k h_k$, $\mu_i\in ({\bf Z}(\overline{H})\cap\bar{\Delta}_{l-1})$, $\nu_p\in \bar{\Delta}_{j-1}$, $\lambda_k\in \bar{\Delta}_{l-j}$, $h_k^{-1}\nu_p\lambda_k h_k\in \bar{\Delta}_{l-1}$, $\psi(\overline{h_k^{-1}\nu_p\lambda_k h_k})=\psi(\overline{\nu_p\lambda_k})$; $g_i$, $f_p$, $h_k\in \varphi(S_\alpha)$.\\
Предположим, найдутся $\nu_p$ такие, что $C\nu_p\notin {\bf Z}(\overline{H})\mod{\bar{\Delta}_j}$ для любого натурального $C$.
Формула (\ref{fm01}) показывает, что $\psi(\overline{\nu_p}\,)\neq 0$ для таких $\nu_p$. Обозначим
\[
f_{p_0}=\max_p\, (f_p\mid\psi(\overline{\nu_p}\,)=M_\nu),
\]
\[
h_{k_0}=\max_k\, (h_k\mid\psi(\overline{\lambda_k}\,)=M_\lambda),
\]
где $M_\nu>0$ и $M_\lambda\geqslant 0$ --- максимальные значения, принимаемые функцией $\psi$ на элементах $\overline{\nu_p}$ и $\overline{\lambda_k}$ соответственно.
Тогда $\psi(\overline{h_{k_0}^{-1}\nu_{p_0}\lambda_{k_0} h_{k_0}}\,)=M_\nu+M_\lambda$ и, так как $\psi(\overline{\mu_i}\,)=0$, получаем противоречие с (\ref{lm6_2_0}).
\end{proof}

\begin{lemma}\label{lm1_2_gr}
Пусть $F$ --- свободная группа с базой $y_1,\ldots,y_n$, $H$ --- нормальная подгруппа группы $F$, $F/H$ --- относительно свободная,
упорядочиваемая группа с базой $y_1H,\ldots,y_nH$, $S=S_\alpha\cup S_\beta$ --- система представителей группы $F$ по подгруппе $H$,
$u\rightarrow \bar{u}$ --- выбирающая функция, $\delta_1,\ldots,\delta_l$, $\mu_1,\ldots,\mu_k$ --- элементы из $S$,
$\delta_iH <\delta_jH$, $\mu_iH <\mu_jH$ при $i <j$.
Тогда если $\{\mu_1^{-1}\mu_1,\ldots,\mu_1^{-1}\mu_k\}\not\subseteq S_\alpha$, то найдутся
$i_0,\,j_0$ такие, что $\overline{\delta_{i_0}\mu_{j_0}}\in S_\beta$ и $\overline{\delta_{i_0}\mu_{j_0}}\neq\overline{\delta_{i}\mu_{j}}$ при $(i_0,j_0)\neq (i,j)$.
\end{lemma}
\begin{proof}
Пусть $\{\mu_1^{-1}\mu_1,\ldots,\mu_1^{-1}\mu_k\}\not\subseteq S_\alpha$.\\
Так как $\overline{\delta_1\mu_1}\neq\overline{\delta_{i}\mu_{j}}$ при $(1,1)\neq (i,j)$, то будем предполагать, что $\overline{\delta_1\mu_1}\in S_\alpha$.
В $F/H$ выберем нормальную подгруппу $B$, порожденную элементом $y_nH$, и подгруппу $A$, порожденную элементами $y_1H,\ldots,y_{n-1}H$. Очевидно, $F/H=AB$, $A\cap B=1$.
Обозначим $\overline{\delta_i{\delta_1}^{-1}}H$ и $\overline{{\mu_1}^{-1}\mu_j}H$ через $b_ia_i$ и $\hat{a}_j\hat{b}_j$ соответственно,
где $a_i,\hat{a}_j\in A$, $b_i,\hat{b}_j\in B$.
Будем считать, для определенности, что в $\{\hat{b}_1,\ldots,\hat{b}_k\}$ есть элементы превосходящие 1.
Обозначим максимальный элемент из $\{b_1,\ldots,b_l\}$ через $x$; максимальный элемент из $\{\hat{b}_1,\ldots,\hat{b}_k\}$ через $z$.
Пусть $b_{i_0}a_{i_0}$ --- максимальный элемент из $\{b_ia_i|b_i=x\}$, $\hat{a}_{j_0}\hat{b}_{j_0}$ --- максимальный элемент из $\{\hat{a}_j\hat{b}_j|\hat{b}_j=z\}$. Элементы $a_i(\overline{\delta_1\mu_1}H)\hat{a}_j$, обозначим их через $t_{ij}$, лежат в $A$, поэтому из $\overline{\delta_{i_0}\mu_{j_0}}H=b_{i_0}t_{i_0 j_0}\hat{b}_{j_0}$; ${\hat{b}_{j_0}}^{t_{i_0 j_0}}>1$;
$b_{i_0}\geqslant 1$ следует
\begin{eqnarray*}
\overline{\delta_{i_0}\mu_{j_0}}H=b_{i_0}{\hat{b}_{j_0}}^{t_{i_0 j_0}}t_{i_0 j_0}\notin A.
\end{eqnarray*}
Из $\overline{\delta_{i_0}\mu_{j_0}}H = \overline{\delta_i\mu_j}H$ вытекает
$t_{i_0 j_0}=t_{ij}$, $b_{i_0}=b_i$, $\hat{b}_{j_0}=\hat{b}_j$; $b_{i_0} a_{i_0}>b_i a_i$, если $i_0\neq i$; $\hat{a}_{j_0} \hat{b}_{j_0} > \hat{a}_j \hat{b}_j$, если $j_0\neq j$; $\overline{\delta_{i_0}\mu_{j_0}}H = b_{i_0} a_{i_0}(\overline{\delta_1\mu_1}H)\hat{a}_{j_0} \hat{b}_{j_0}>b_i a_i(\overline{\delta_1\mu_1}H)\hat{a}_j \hat{b}_j=\overline{\delta_i\mu_j}H$, если $(i_0,j_0)\neq (i,j)$. Следовательно,
\begin{eqnarray*}
\overline{\delta_{i_0}\mu_{j_0}}H\neq \overline{\delta_i\mu_j}H,
\end{eqnarray*}
если $(i_0,j_0)\neq (i,j)$.
\end{proof}

\begin{lemma}\label{lm2_2_gr}
Пусть $X$ --- свободная группа с базой $\{x_j | j\in J\}$, $X_n$ --- $n$-й член нижнего центрального ряда группы $X$,
$\mathfrak{X}$ --- фундаментальный идеал кольца ${\bf Z}(X)$, $v\in X$. Тогда и только тогда $v\in X_n\setminus X_{n+1}$, когда $D_j(v)\in \mathfrak{X}^{n-1} \, (j\in J)$
и найдется элемент $j_0\in J$ такой, что $D_{j_0}(v)\in \mathfrak{X}^{n-1}\setminus  \mathfrak{X}^n$.
\end{lemma}
\begin{proof}
Известно \cite{Fx}, что базу ${\bf Z}$-модуля $\mathfrak{X}^k/\mathfrak{X}^{k+1}$ образуют элементы
вида $(x_{j_1}-1)\ldots (x_{j_k}-1)+\mathfrak{X}^{k+1}$ и что $v\in X_n$ тогда и только тогда, когда $v-1\in \mathfrak{X}^n$.
Следовательно, утверждение леммы вытекает из равенства\\
$v-1=\sum_{j\in J} (x_j-1)D_j(v)$.
\end{proof}

\begin{lemma}\label{tm0_gr}
Пусть $X$ --- свободная группа с базой $\{x_j | j\in J\}$,
$\mathfrak{X}$ --- фундаментальный идеал кольца ${\bf Z}(X)$, $v\in X$, $K\subseteq J$, $F_K=\text{гр }(x_j| j\in K)$. Тогда
$v$ удовлетворяет условиям
\begin{eqnarray}\label{0tm02_gr}
D_k(v)\equiv ~0\mod{\mathfrak{X}^n},~k\in J\setminus K;~D_k(v)\in{\bf Z}(F_K)\mod{\mathfrak{X}^n},~k\in  K
\end{eqnarray}
если и только если $v\in \text{гр }(F_K,\,\gamma_{n+1}X)$.
\end{lemma}
\begin{proof}
Лемма \ref{lm2_2_gr} показывает, что из $v\in \text{гр }(F_K,\,\gamma_{n+1}X)$ следуют сравнения (\ref{0tm02_gr}).
Необходимо доказать обратное.\\
Пусть $\varphi\text{: }{\bf Z}(X)\rightarrow {\bf Z}(X)$ --- эндоморфизм, определяемый отображением
$x_j\rightarrow x_j$ при $j\in K$, $x_j\rightarrow 1$ при $j\in J\setminus K$. Обозначим $\bar{v}=\varphi(v)$.
Ясно, что $\bar{v}\in F_K$ и $D_k(v)\equiv D_k(\bar{v})\mod{\mathfrak{X}^n},~k\in  K$.
Так как $D_k(v\bar{v}^{-1})=D_k(v)\bar{v}^{-1}-D_k(\bar{v})\bar{v}^{-1}$, то
\begin{eqnarray*}
D_k(v\bar{v}^{-1})\equiv ~0\mod{\mathfrak{X}^n},~k\in J,
\end{eqnarray*}
т.е. $v\bar{v}^{-1}\in \gamma_{n+1}X$ (лемма \ref{lm2_2_gr}).
\end{proof}

\begin{lemma}\label{lm5_2_gr}\cite{Fx}
Пусть $G$ --- свободная группа с базой с базой $\{g_j | j\in J\}$, $\{D_j | j\in J\}$ --- соответствующие этой базе производные Фокса кольца ${\bf Z}(G)$. Пусть, далее, $H$ --- подгруппа группы $G$, $f\in {\bf Z}(H)$,
$\{h_i | i\in I\}$ --- база $H$, $\{\partial_i | i\in I\}$ --- соответствующие этой базе производные Фокса кольца ${\bf Z}(H)$.\\
Тогда $D_j(f)=\sum_k\partial_k(f) D_j(h_k)$.
\end{lemma}
\begin{lemma}\label{lm3_2_gr}
Пусть $X$ --- свободная группа с базой $\{x_j | j\in {\bf N}\}$, $X_n$ --- $n$-й член нижнего центрального ряда группы $X$,
$\mathfrak{X}$ --- фундаментальный идеал кольца ${\bf Z}(X)$, $v\in X$. Если $D_1(v)\notin \mathfrak{X}^{j-1}$, то
$D_1([v,x_2])\notin \mathfrak{X}^j$. 
\end{lemma}
\begin{proof}
Непосредственная проверка показывает, что
\begin{eqnarray*}
D_1([v,x_2])=D_1(v)(x_2 - 1) + D_1(v)(1-v^{-1}{x_2}^{-1}vx_2).
\end{eqnarray*}
Так как $1-v^{-1}{x_2}^{-1}vx_2\in \mathfrak{X}^2$ \cite{Fx}, то $D_1([v,x_2])\notin \mathfrak{X}^j$.
\end{proof}
\noindent Пусть $G$ --- свободная группа, $H$ --- подгруппа группы $G$, ряд
\begin{eqnarray*}
G=G_1\geqslant\ldots \geqslant G_l\geqslant\ldots
\end{eqnarray*}
нормальный, с абелевыми факторами без кручения и $[G_p,G_m\,]\leqslant G_{p+m}$.\\
Положим $H_t=H\cap G_t$, $\Delta_i$ --- идеал, порожденный $(G_{i_1} - 1)\cdots (G_{i_s} - 1)$ в ${\bf Z}(G)$,
$\Delta_i^\prime$ --- идеал, порожденный $(H_{i_1} - 1)\cdots (H_{i_s} - 1)$ в ${\bf Z}(H)$, где $i_1 + \cdots + i_s\geqslant i$.\\
Нетрудно видеть, что
\begin{gather}
{\bf Z}(H)\cap \Delta_i\subseteq \Delta_i^\prime \mod{G_m};\notag\\
\text{если } u\in \Delta_i\setminus \Delta_{i+1} \mod{G_m},~v\in \Delta_j\setminus \Delta_{j+1} \mod{G_m}, \text{ то }\notag\\
uv\in \Delta_{i+j}\setminus \Delta_{i+j+1} \mod{G_m}.\notag
\end{gather}
\begin{proposition}\label{tm4_gr}
Пусть $F$ --- свободная группа с базой $y_1,\ldots,y_n$, $n\geqslant 3$, $N$ --- вербальная подгруппа группы $F$,
\begin{eqnarray}\label{tm4_0_gr}
N=N_{11} \geqslant \ldots \geqslant N_{1,m_1+1}=N_{21} \geqslant \ldots \geqslant N_{s,m_s+1},
\end{eqnarray}
где $N_{kl}$ --- l-й член нижнего центрального ряда группы $N_{k1}$.
Пусть, далее, $R$ --- нормальная подгруппа группы $F$, $R\leqslant N$, $H = \textup{гр }(y_1,\ldots,y_{n-1})$.\\
Если $H\cap RN_{1j}\neq H\cap N_{1j}$, то $H\cap RN_{kl}\neq H\cap N_{kl}\, (N_{kl}\leqslant N_{1j})$.
\end{proposition}
\begin{proof}
Отметим, что $H\cap N_{kl}=\gamma_l(H\cap N_{k1})$.\\
Предположим, $H\cap RN_{1j}\neq H\cap N_{1j}$. Покажем, что
\begin{eqnarray}\label{pr1_gr}
H\cap RN_{1l}> \gamma_l(H\cap N)\,~(N_{1l}\leqslant N_{1j}).
\end{eqnarray}
По условию, $H\cap RN_{1j}> \gamma_j(H\cap N)$.
Пусть для всех членов ряда~(\ref{tm4_0_gr}) от $N_{1j}$ до $N_{1l}$ включительно формула~(\ref{pr1_gr}) верна ($l\geqslant j$).
Обозначим через $B$ группу $H\cap N$ и через $\mathfrak{X}$ фундаментальный идеал кольца ${\bf Z}(B)$.
Выберем в группе $B$
базу $\{x_z | z\in {\bf N}\}$, $\{\partial_z | z\in {\bf N}\}$ --- соответствующие этой базе
производные Фокса кольца ${\bf Z}(B)$.
Пусть $v\in (H\cap RN_{1l})\setminus \gamma_l(B)$.
Мы можем и будем считать, что
$\partial_1(v)\notin \mathfrak{X}^{l-1}$.
Полагаем $w=[v,x_2]$.
По лемме~\ref{lm3_2_gr}, $\partial_1(w)\notin \mathfrak{X}^l$, т.е. $w\in (H\cap RN_{1,l+1})\setminus \gamma_{l+1}(B)$.
Теперь соображения индукции заканчивают доказательство формулы~(\ref{pr1_gr}).
Из (\ref{pr1_gr}) следует $H\cap RN_{21}> H\cap N_{21}$.\\
Остается заметить, что из $H\cap RN_{k1}> H\cap N_{k1}$ следует\\
$\gamma_{l}(H\cap RN_{k1})> \gamma_{l}(H\cap N_{k1})~(l\in {\bf N})$, т.е.
$H\cap RN_{kl}> H\cap N_{kl}\, (N_{kl}\leqslant N_{21})$.
\end{proof}

\begin{proposition}\label{tm2_gr}
Пусть $F$ --- свободная группа с базой $y_1,\ldots,y_n$, $N$ --- вербальная подгруппа группы $F$, $F/N$ --- правоупорядочиваемая группа, \begin{eqnarray}\label{tm2_0_gr}
N=N_{11} \geqslant \ldots \geqslant N_{1,m_1+1}=N_{21} \geqslant \ldots \geqslant N_{s,m_s+1},
\end{eqnarray}
где $N_{kl}$ --- l-й член нижнего центрального ряда группы $N_{k1}$.\\
Пусть, далее, $r\in N_{1,i}\backslash N_{1,i+1}\,(i\leqslant m_1)$, $R$ --- нормальная подгруппа, порожденная в группе $F$ элементом $r$,
$H = \textup{гр }(y_1,\ldots,y_{n-1})$.\\
Если $H\cap RN_{21}=H\cap N_{21}$, то $H\cap RN_{kl}=H\cap N_{kl}\,(k> 1)$.
\end{proposition}
\begin{proof}
\noindent Отметим, что $H\cap N_{kl}=\gamma_l(H\cap N_{k1})$.
Построим индукцией подгруппы $\sqrt{RN_{kl}}$ группы $F$.
Подгруппа $\sqrt{RN_{21}}$ --- множество всех элементов группы $F$, попадающих в некоторой степени в $RN_{21}$. Пусть построена подгруппа $\sqrt{RN_{k1}}$. Положим $\sqrt{RN_{kl}}$ --- множество всех элементов группы $F$, попадающих в некоторой степени в $\gamma_l (\sqrt{RN_{k1}})R$.
Ясно, что $\sqrt{RN_{kl}}$ --- нормальная подгруппа группы $F$ и $F/\sqrt{RN_{kl}}$ --- правоупорядочиваемая группа. Пусть $D_1,\ldots,D_n$ --- производные Фокса
кольца ${\bf Z}(F)$. Докажем, что $D_n(r)\not\equiv 0\mod \sqrt{RN_{21}}$. Предположим противное.
Тогда по теореме~\ref{tm1_gr} нашлись бы $v_1,\ldots,v_d$ из $H\cap \sqrt{RN_{21}}$; $f_1,\ldots,f_d$ из $F$
такие, что
\begin{eqnarray*}
r\equiv v_1^{f_1}\cdots v_d^{f_d}\mod [\sqrt{RN_{21}},\sqrt{RN_{21}}\,].
\end{eqnarray*}
Каждый из $v_1,\ldots,v_d$ попадает в некоторой степени в $N_{21}$, поэтому все эти элементы принадлежат $N_{21}$. Следовательно, $r\in N_{1,i+1}$, что противоречит условию предложения~\ref{tm2_gr}.\\
Очевидно, предложение~\ref{tm2_gr} будет доказано, если мы покажем, что
\begin{eqnarray}\label{tm2_0_0_gr}
H\cap \sqrt{RN_{kl}}=H\cap N_{kl}.
\end{eqnarray}
По условию, $H\cap \sqrt{RN_{21}}=H\cap N_{21}$.
Пусть формула (\ref{tm2_0_0_gr}) справедлива для всех членов ряда (\ref{tm2_0_gr}) от $N_{21}$ до $N_{kl}$
включительно $(N_{21} \geqslant N_{kl})$. Требуется доказать\\
а) $H\cap \sqrt{RN_{{k+1},2}}=H\cap N_{{k+1},2}\,(l= m_k+1,\,k\neq s)$,\\
б) $H\cap \sqrt{RN_{k,{l+1}}}=H\cap N_{k,{l+1}}\,(l\leqslant m_k)$.\\
Докажем а). Будем иметь $H\cap \sqrt{RN_{k+1,1}}=H\cap N_{k+1,1}$.\\
Выберем $v\in H\cap \sqrt{RN_{{k+1},2}}$. Найдутся $\alpha\in {\bf Z}(F)$ и $c\in {\bf N}$ такие, что
\begin{eqnarray}\label{tm2_1_gr}
D_m(v^c)\equiv D_m(r)\cdot\alpha\mod\sqrt{RN_{{k+1},1}},~m\in \{1,\ldots,n\}.
\end{eqnarray}
Так как в групповом кольце правоупорядочиваемой группы нет делителей нуля, то из $D_n(r)\not\equiv 0\mod \sqrt{RN_{21}}$, $D_n(v)=0$ и (\ref{tm2_1_gr})
следует, что
\begin{eqnarray}\label{tm2_1_1_gr}
\alpha\equiv 0\mod \sqrt{RN_{{k+1},1}}\,.
\end{eqnarray}
Формулы (\ref{tm2_1_gr}), (\ref{tm2_1_1_gr}) показывают, что $D_m(v^c)\equiv 0\mod N_{k+1,1},~m\in \{1,\ldots,n\}$, поэтому
$v\in H\cap N_{{k+1},2}$.\\
Докажем б). В группе $\sqrt{RN_{k1}}$ рассмотрим ряд
\begin{eqnarray*}
\sqrt{RN_{k1}}\geqslant\ldots \geqslant\sqrt{RN_{kl}}\geqslant\ldots.
\end{eqnarray*}
Непосредственно проверяется, что $[\sqrt{RN_{kp}},\sqrt{RN_{km}}\,]\leqslant \sqrt{RN_{k,p+m}}$.\\
Пусть $v\in H\cap \sqrt{RN_{k,l+1}}$, $S=S_\alpha\cup S_\beta$ --- система представителей группы $F$ по подгруппе $\sqrt{RN_{k1}}$. Выберем в группе $\sqrt{RN_{k1}}$ базу $\{x_z | z\in {\bf N}\}$, состоящую из $\alpha$ и $\beta$-порождающих, $\{\partial_z | z\in {\bf N}\}$ --- соответствующие этой базе производные Фокса кольца ${\bf Z}(\sqrt{RN_{k1}}\,)$.
Найдутся $u\in \gamma_{l+1} (\sqrt{RN_{k1}})$, $k_i\in {\bf Z}(\sqrt{RN_{k1}}\,)$, $\mu_i\in S$, $i\in \{1,\ldots,d\}$,
$k\in {\bf N}$ такие, что
\begin{eqnarray}\label{tm2_3_gr}
D_m(v^k)\equiv D_m(r)\cdot  \sum_{i={1}}^{d} k_i\mu_i+\sum_{z\in {\bf N}} \partial_z(u)D_m(x_z)\mod\sqrt{RN_{kl}},
\end{eqnarray}
$m\in \{1,\ldots,n\}$. Будем иметь
\begin{eqnarray}\label{tm2_4_gr}
0\equiv D_n(r)\cdot  \sum_{i={1}}^{d} k_i\mu_i+ \sum_{z\in {\bf N}} \partial_z(u)D_n(x_z)\mod\sqrt{RN_{kl}}.
\end{eqnarray}
Обозначим через $\mathfrak{X}$ фундаментальный идеал кольца ${\bf Z}(\sqrt{RN_{k1}}\,)$.\\
Из $D_n(r)\not\equiv 0\mod \sqrt{RN_{k1}}$ следует, что
\begin{eqnarray*}
D_n(r) = \sum_{i={1}}^p  t_i \delta_i,
\end{eqnarray*}
где $t_i\in {\bf Z}(\sqrt{RN_{k1}}\,)$, $\delta_i\in S$, $i\in \{1,\ldots,p\}$ и
не все элементы из $\{t_1,\ldots,t_p\}$ принадлежат $\mathfrak{X}$.
Пусть $\{\delta_{j_1},\ldots,\delta_{j_b}\}$ состоит из представителей, у которых
$t_{j_1},\ldots,t_{j_b}$ не принадлежат $\mathfrak{X}$.
Если не все $k_i$ принадлежат $\mathfrak{X}^l\mod\sqrt{RN_{kl}}$, то выберем минимальное $l_0$ такое, что $k_i\in \mathfrak{X}^{l_0}\mod\sqrt{RN_{kl}}$, $i=1,\ldots,d$
и пусть $\{\mu_{i_1},\ldots,\mu_{i_a}\}$ состоит из представителей, у которых
$\{k_{i_1},\ldots,k_{i_a}\}\subseteq \mathfrak{X}^{l_0}\setminus\mathfrak{X}^{l_0 +1}\mod\sqrt{RN_{kl}}$.\\
Группа $F/\sqrt{RN_{k1}}$ - правоупорядочиваема, поэтому
среди элементов $\delta_{j_h}\mu_{i_f}$ найдется элемент
$\delta_{j_0}\mu_{i_0}$ такой, что $\delta_{j_0}\mu_{i_0}\sqrt{RN_{k1}}\neq\delta_{j_h}\mu_{i_f}\sqrt{RN_{k1}}$
при $(j_0,i_0)\neq (j_h,i_f)$.
По лемме \ref{lm2_2_gr}, все $\partial_z(u)$ лежат в $\mathfrak{X}^l$,
поэтому, из (\ref{tm2_4_gr}) вытекает существование $c_{i_0}\in {\bf N}$ такого, что
$c_{i_0} k_{i_0}\in \mathfrak{X}^{l_0 + 1}\mod\sqrt{RN_{kl}}$. Продолжая этот процесс применительно
к элементу $v^{k{c_{i_0}}}$, мы найдем в конце концов число $c_l \in {\bf N}$ такое, что
$c_l k_i\in \mathfrak{X}^l\mod\sqrt{RN_{kl}}$, $i\in \{1,\ldots,d\}$.
Следовательно, (\ref{tm2_3_gr}) можно переписать в виде
\begin{eqnarray}\label{tm2_6_gr}
D_m(v^c)\equiv \sum_{i={1}}^{d_m} \mu_{i,m} g_{i,m} \mod\sqrt{RN_{kl}},~m\in \{1,\ldots,n\},
\end{eqnarray}
где $c=kc_l$, $\mu_{i,m}\in \mathfrak{X}^l$, $g_{i,m}\in S$.\\
Если $x_z=Ky_m\overline{Ky_m}^{\,-1}$, $K\in S$, то ввиду  следствия \ref{lm4_2_gr} и леммы \ref{lm5_2_gr}
\begin{eqnarray}\label{tm2_7_gr}
D_m(v^c)= \partial_z(v^c)\overline{Ky_m}^{\,-1}+\sum_{i={1}}^g \lambda_i t_i^{-1},
\end{eqnarray}
где $\lambda_i\in {\bf Z}(\sqrt{RN_{k1}}\,)$, $t_i\in S$, $t_i\neq \overline{Ky_m}$, $i\in \{1,\ldots,g\}$.
Из (\ref{tm2_6_gr}), (\ref{tm2_7_gr}) получаем
\begin{eqnarray}\label{tm2_8_gr}
\partial_z(v^c)\in \mathfrak{X}^l\mod\sqrt{RN_{kl}},~z\in {\bf N}.
\end{eqnarray}
Положим $H_t=H\cap \sqrt{RN_{kt}}$, $t\in \{1,\ldots,l\}$.
Обозначим через $\Delta_l$ идеал, порожденный $(H_{i_1} - 1)\cdots (H_{i_s} - 1)$ в ${\bf Z}(H_1)$,  где $i_1 + \cdots + i_s\geqslant l$. Формула (\ref{tm2_8_gr}) показывает, что $\partial_z(v^c)\in \Delta_l\mod\sqrt{RN_{kl}},~z\in {\bf N}$.\\
По условию, $H\cap \sqrt{RN_{kt}}=H\cap N_{kt}$, следовательно
\begin{eqnarray}\label{tm2_9_gr}
H_t = \gamma_t (H\cap N_{k1}), ~t\in \{1,\ldots,l\}.
\end{eqnarray}
Обозначим через $\mathfrak{X}^\prime$ фундаментальный идеал кольца ${\bf Z}(H\cap N_{k1})$.
Известно \cite{Fx}, что
\begin{eqnarray}\label{tm2_11_gr}
\gamma_t (H\cap N_{k1})-1\subseteq (\mathfrak{X}^\prime)^t.
\end{eqnarray}
Из (\ref{tm2_9_gr}), (\ref{tm2_11_gr}) следует, что $\Delta_l \subseteq (\mathfrak{X}^\prime)^l \mod\sqrt{RN_{kl}}$,
т.е.
\begin{eqnarray*}
\partial_z(v^c)\in (\mathfrak{X}^\prime)^l\mod\sqrt{RN_{kl}},~z\in {\bf N}.
\end{eqnarray*}\\
Так как $H\cap \sqrt{RN_{kl}}=H\cap N_{kl}=\gamma_l (H\cap N_{k1})$, то $\partial_z(v^c)\in (\mathfrak{X}^\prime)^l,~z\in {\bf N}$.
Это означает, что $v^c-1\in (\mathfrak{X}^\prime)^{l+1}$, следовательно $v^c\in \gamma_{l+1} (H\cap N_{k1})$ \cite{Fx},
т.е. $v\in H\cap N_{k,l+1}$.
\end{proof}

\begin{proposition}\label{tm5_gr}
Пусть $F$ --- свободная группа с базой $y_1,\ldots,y_n$, $N$ --- вербальная подгруппа группы $F$, $F/N$ --- упорядочиваемая группа,
\begin{eqnarray}\label{tm5_0_gr}
N=N_{11} \geqslant \ldots \geqslant N_{1,m_1+1}=N_{21} \geqslant \ldots \geqslant N_{s,m_s+1},
\end{eqnarray}
где $N_{kl}$ --- l-й член нижнего центрального ряда группы $N_{k1}$.\\
Пусть, далее, $r\in N_{1i}\backslash N_{1,i+1}\,(i\leqslant m_1)$, $R$ --- нормальная подгруппа, порожденная в группе $F$ элементом $r$, $H = \textup{гр }(y_1,\ldots,y_{n-1})$.\\
Если элемент $r$ не сопряжен по модулю $N_{1,i+1}$ ни с каким словом от $y_1^{\pm 1},\ldots,y_{n-1}^{\pm 1}$, то
$H\cap RN_{1l}=H\cap N_{1l}$.
\end{proposition}
\begin{proof}
\noindent Отметим, что $H\cap N_{1l}=\gamma_l(H\cap N)$.
Положим $\sqrt{RN_{1l}}$ --- множество всех элементов группы $F$, попадающих в некоторой степени в $RN_{1l}$.
Ясно, что $\sqrt{RN_{1l}}$ --- нормальная подгруппа группы $F$ и $N/\sqrt{RN_{1l}}$ --- нильпотентная группа без кручения.
Очевидно, предложение~\ref{tm5_gr} будет доказано, если мы покажем, что
\begin{eqnarray}\label{tm3_0_0_gr}
H\cap \sqrt{RN_{1l}}=H\cap N_{1l}.
\end{eqnarray}
При $l=i$ формула (\ref{tm3_0_0_gr}), очевидно, справедлива.
Пусть формула (\ref{tm3_0_0_gr}) справедлива для всех членов ряда (\ref{tm5_0_gr}) от $N_{1i}$ до $N_{1l}$
включительно $(N_{1i} \geqslant N_{1l})$. Требуется доказать
$H\cap \sqrt{RN_{1,{l+1}}}=H\cap N_{1,{l+1}}\,(l\leqslant m_1)$.\\
В группе $N$ рассмотрим ряд
\begin{eqnarray*}
N=N_{11}\geqslant\ldots N_{1i}\geqslant\sqrt{RN_{1,i+1}}\geqslant\ldots \geqslant\sqrt{RN_{1l}}\geqslant\ldots.
\end{eqnarray*}
Непосредственно проверяется, что $[\sqrt{RN_{1p}},\sqrt{RN_{1m}}\,]\leqslant \sqrt{RN_{1,p+m}}$.\\
Пусть $D_1,\ldots,D_n$ --- производные Фокса
кольца ${\bf Z}(F)$, $S=S_\alpha\cup S_\beta$ --- система представителей группы $F$ по подгруппе $N$, $v\in H\cap \sqrt{RN_{1,l+1}}$. Выберем в группе $N$ базу $\{x_z | z\in {\bf N}\}$, состоящую из $\alpha$ и $\beta$-порождающих, $\{\partial_z | z\in {\bf N}\}$ --- соответствующие этой базе производные Фокса кольца ${\bf Z}(N\,)$.
Найдутся $u\in N_{1,l+1}$, $k_p\in {\bf Z}(N\,)$, $\mu_p\in S$, $p\in \{1,\ldots,d\}$, $k\in {\bf N}$
такие, что
\begin{eqnarray}\label{tm3_3_gr}
D_m(v^k)\equiv \sum_{z\in {\bf N}} \partial_z(u)D_m(x_z) + D_m(r)\cdot  \sum_{p={1}}^{d} k_p\mu_p\mod\sqrt{RN_{1l}},
\end{eqnarray}
$m\in \{1,\ldots,n\}$. Значит,
\begin{eqnarray}\label{tm3_4_gr}
\sum_{z\in {\bf N}} \partial_z(u)D_n(x_z) + D_n(r)\cdot  \sum_{p={1}}^{d} k_p\mu_p\equiv 0\mod\sqrt{RN_{1l}}.
\end{eqnarray}
Обозначим через $\mathfrak{X}$ фундаментальный идеал кольца ${\bf Z}(N)$.
По лемме \ref{lm2_2_gr}, $\partial_z(u)\in\mathfrak{X}^l$, $\partial_z(v)\in\mathfrak{X}^{l-1}$, $\partial_z(r)\in\mathfrak{X}^{i-1}$ $(z\in {\bf N})$.
Пусть $\Delta_t$ --- идеал, порожденный в ${\bf Z}(N)$ произведениями $(\sqrt{RN_{1i_1}} - 1)\cdots (\sqrt{RN_{1i_s}} - 1)$,
$i_1 + \cdots + i_s\geqslant t$.\\
Рассмотрим случай $\{\partial_z(v^k) | z\in {\bf N}\} \not\subseteq \Delta_l$.\\
Так как $r\in N_{1i}\setminus  N_{1,i+1}$, то найдется $z_0$, такое, что $\partial_{z_0}(r) \in \Delta_{i-1}\setminus \Delta_i$. Если $x_{z_0}=Ky_t\overline{Ky_t}^{\,-1}$, $K\in S$, то
ввиду  следствия \ref{lm4_2_gr} и леммы \ref{lm5_2_gr}
\begin{eqnarray*}
D_t(r)=\sum_{z\in {\bf N}} \partial_z(r)D_t(x_z)= \partial_{z_0}(r)\overline{Ky_t}^{\,-1}+\sum_{p=1}^g \lambda_p t_p^{-1},
\end{eqnarray*}
где $\lambda_p\in \Delta_{i-1}$, $t_p\in S$, $t_p\neq \overline{Ky_t}$, $p\in \{1,\ldots,g\}$,
т.е.
\begin{eqnarray}\label{tm3_4_0_gr}
D_t(r) = a_1\delta_1+ \cdots +a_q\delta_q,
\end{eqnarray}
$\delta_k\in S$, $\delta_k\neq \delta_p$ при $k\neq p$, $a_p \in \Delta_{i-1}$ и $a_p \in \Delta_{i-1}\setminus \Delta_i$ для некоторых $p$.
В рассматриваемом случае из (\ref{tm3_3_gr}), (\ref{tm3_4_0_gr}) следует $k_p \in \Delta_{l-i}$, $p\in \{1,\ldots,d\}$ и
$k_p \in \Delta_{l-i}\setminus \Delta_{l-i+1}$ для некоторых $p$.
Пусть $M$ --- подмножество в $\{1,\ldots,d\}$ такое, что $k_p \in \Delta_{l-i}\setminus \Delta_{l-i+1}$ при $p\in M$,
$N$ --- подмножество в $\{1,\ldots,q\}$ такое, что $a_p \in \Delta_{i-1}\setminus \Delta_i$ при $p\in N$.\\
Так как $D_m(v)$ --- сумма элементов вида $b\theta$, где $\theta$ --- $\alpha$-представитель, $b\in N\cap H$,
то (\ref{tm3_3_gr}) показывает, что не существует $k_0,p_0$  таких, что $\overline{\delta_{k_0}\mu_{p_0}}\in S_\beta$ и $\overline{\delta_{k_0}\mu_{p_0}}\neq\overline{\delta_{k}\mu_{p}}$ при  $(k_0,p_0)\neq (k,p)$, $k_0, k \in M$, $p_0, p \in N$.\\
Из леммы \ref{lm1_2_gr} следует существование $\mu\in S$ такого, что $\overline{\mu^{-1}\mu_p}\in S_\alpha$, $p \in M$, т.е.
\begin{eqnarray}\label{tm3_4_1_1_gr}
\mu^{-1}\cdot  \sum_{p={1}}^{d} k_p\mu_p\equiv b_1\hat{\mu}_1+ \cdots +b_{\hat{d}}\hat{\mu}_{\hat{d}}\mod{{\bf Z}(F)\Delta_{l-i+1}},
\end{eqnarray}
$\hat{\mu}_k\in S_\alpha$, $\hat{\mu}_k\neq \hat{\mu}_p$ при $k\neq p$, $b_k\in \Delta_{l-i}\setminus \Delta_{l-i+1}$.
Покажем, что если $x_z$ --- $\beta$-порождающий, то
\begin{eqnarray}\label{tm3_4_1_gr}
\partial_z(r^\mu)\in\mathfrak{X}^i.
\end{eqnarray}
Предположим противное. Значит, $\partial_z(r^\mu)\not\in \Delta_i$.\\
Рассмотрим случай $x_z=Ky_n\overline{Ky_n}^{\,-1}$, $K\in S$.
Тогда
\begin{eqnarray*}
D_n(r^\mu) = a_1\delta_1+ \cdots +a_q\delta_q,
\end{eqnarray*}
$\delta_k\in S$, $\delta_k\neq \delta_p$ при $k\neq p$, $\{a_1,\ldots,a_q\}\subset {\bf Z}(N)$, $a_p \in \Delta_{i-1}\setminus \Delta_i$ для некоторых $p$. Из упорядочиваемости группы $F/N$, следует, что
\begin{eqnarray}\label{tm3_4_2_gr}
D_n(r^\mu)\mu^{-1}\cdot  \sum_{p={1}}^{d} k_p\mu_p\equiv \hat{a}_1\hat{\delta}_1+ \cdots +\hat{a}_s\hat{\delta}_s\mod\sqrt{RN_{1l}},
\end{eqnarray}
$\hat{\delta}_k\in S$, $\hat{\delta}_k\neq \hat{\delta}_p$ при $k\neq p$, $\{\hat{a}_1,\ldots,\hat{a}_s\}\subset {\bf Z}(N)$ и $\hat{a}_p \in \Delta_{l-1}\setminus \Delta_l$ для некоторых $p$.
Формула (\ref{tm3_4_2_gr}) противоречит (\ref{tm3_4_gr}).\\
Рассмотрим случай $x_z=Ky_t\overline{Ky_t}^{\,-1}$, $t<n$, $K\in S_\beta$. Тогда
\begin{eqnarray*}
D_t(r^\mu) = a_1\delta_1+ \cdots +a_q\delta_q,
\end{eqnarray*}
$\delta_k\in S$, $\delta_k\neq \delta_p$ при $k\neq p$, $\{a_1,\ldots,a_q\}\subset {\bf Z}(N)$ и, для некоторых $p$, $\delta_p\in S_\beta$, $a_p \in \Delta_{i-1}\setminus \Delta_i$. Из упорядочиваемости группы $F/N$, следует, что
\begin{eqnarray}\label{tm3_4_3_gr}
D_t(r^\mu)\mu^{-1}\cdot  \sum_{p={1}}^{d} k_p\mu_p\equiv \hat{a}_1\hat{\delta}_1+ \cdots +\hat{a}_s\hat{\delta}_s\mod\sqrt{RN_{1l}},
\end{eqnarray}
$\hat{\delta}_k\in S$, $\hat{\delta}_k\neq \hat{\delta}_p$ при $k\neq p$, $\{\hat{a}_1,\ldots,\hat{a}_s\}\subset {\bf Z}(N)$ и, для некоторых $p$, $\hat{\delta}_p\in S_\beta$, $\hat{a}_p \in \Delta_{l-1}\setminus \Delta_l$.
Формула (\ref{tm3_4_3_gr}) противоречит (\ref{tm3_3_gr}). Полученные противоречия доказывают формулу (\ref{tm3_4_1_gr}).\\
Рассмотрим случай $x_z=Ky_t\overline{Ky_t}^{\,-1}$, $t<n$, $K\in S_\alpha$, $\partial_z(r^\mu)\not\in\mathfrak{X}^i$. Тогда
\begin{eqnarray}\label{tm3_4_1_2_gr}
D_t(r^\mu)\equiv a_1\delta_1+ \cdots +a_q\delta_q\mod{{\bf Z}(F)\Delta_i},
\end{eqnarray}
$\delta_k\in S_\alpha$, $\delta_k\neq \delta_p$ при $k\neq p$, $a_k\in \Delta_{i-1}\setminus \Delta_i$, $a_1=\partial_z(r^\mu)$.\\
Из (\ref{tm3_3_gr}), (\ref{tm3_4_1_1_gr}), (\ref{tm3_4_1_2_gr})следует
\begin{eqnarray*}
D_t(v^k)\equiv (a_1\delta_1+ \cdots +a_q\delta_q)\cdot  (b_1\hat{\mu}_1+ \cdots +b_{\hat{d}}\hat{\mu}_{\hat{d}})\mod{{\bf Z}(F)\Delta_l},
\end{eqnarray*}
$\delta_k\in S_\alpha$, $\delta_k\neq \delta_p$ при $k\neq p$, $a_k\in \Delta_{i-1}\setminus \Delta_i$, $a_1=\partial_z(r^\mu)$, $\hat{\mu}_k\in S_\alpha$, $\hat{\mu}_k\neq \hat{\mu}_p$ при $k\neq p$, $b_k\in \Delta_{l-i}\setminus \Delta_{l-i+1}$.
Тогда, ввиду леммы \ref{lm6_2_gr}, найдется натуральное $C$ такое, что для всех $\alpha$-порождающих $x_z$
\begin{eqnarray}\label{tm3_4_1_0_gr}
C\partial_z(r^\mu)\in{\bf Z}(N\cap H)\mod{\mathfrak{X}^i}.
\end{eqnarray}
Из (\ref{tm3_4_1_gr}), (\ref{tm3_4_1_0_gr}), ввиду леммы \ref{tm0_gr}, следует, что $r^{C\mu}\in HN_{1,i+1}$.
Так как $N/N_{1,i+1}$ --- группа с однозначным извлечением корней, то $r^\mu\in HN_{1,i+1}$. Противоречие.\\
Рассмотрим случай $\{\partial_z(v^k) | z\in {\bf N}\} \subseteq \Delta_l$.\\
Положим $H_t=H\cap \sqrt{RN_{1t}}$, $t\in \{1,\ldots,l\}$.
Обозначим через $\Delta_l^\prime$ идеал, порожденный $(H_{i_1} - 1)\cdots (H_{i_s} - 1)$ в ${\bf Z}(H_1)$,  где $i_1 + \cdots + i_s\geqslant l$.\\
Тогда $\{\partial_z(v^k) | z\in {\bf N}\} \subseteq \Delta_l^\prime\mod\sqrt{RN_{1l}}$.
По условию, $H\cap \sqrt{RN_{1t}}=H\cap N_{1t}$, следовательно
\begin{eqnarray}\label{tm2_9_gr0}
H_t = \gamma_t (H\cap N), ~t\in \{1,\ldots,l\}.
\end{eqnarray}
Обозначим через $\mathfrak{X}^\prime$ фундаментальный идеал кольца ${\bf Z}(H\cap N)$.
Известно \cite{Fx}, что
\begin{eqnarray}\label{tm2_11_gr0}
\gamma_t (H\cap N)-1\subseteq (\mathfrak{X}^\prime)^t.
\end{eqnarray}
Из (\ref{tm2_9_gr0}), (\ref{tm2_11_gr0}) следует, что $\Delta_l^\prime \subseteq (\mathfrak{X}^\prime)^l$,
т.е.
\begin{eqnarray*}
\partial_z(v^k)\in (\mathfrak{X}^\prime)^l\mod\sqrt{RN_{1l}},~z\in {\bf N}.
\end{eqnarray*}\\
Так как $H\cap \sqrt{RN_{1l}}=H\cap N_{1l}=\gamma_l (H\cap N)$, то $\partial_z(v^k)\in (\mathfrak{X}^\prime)^l,~z\in {\bf N}$.
Это означает, что $v^k-1\in (\mathfrak{X}^\prime)^{l+1}$, следовательно $v^k\in \gamma_{l+1} (H\cap N)$ \cite{Fx},
т.е. $v\in H\cap N_{1,l+1}$.
\end{proof}
\noindent Из предложений \ref{tm4_gr}, \ref{tm2_gr}, \ref{tm5_gr}, вытекает
\begin{theorem}
Пусть $F$ --- свободная группа с базой $y_1,\ldots,y_n$ $(n>2)$, $N$ --- вербальная подгруппа группы $F$, $F/N$ --- упорядочиваемая группа,
\begin{eqnarray}\label{end_gr}
N=N_{11} \geqslant \ldots \geqslant N_{1,m_1+1}=N_{21} \geqslant \ldots \geqslant N_{s,m_s+1},
\end{eqnarray}
где $N_{ij}$ --- j-й член нижнего центрального ряда группы $N_{i1}$, $s \geqslant 1$.\\
Пусть, далее, $r\in N_{1,i}\backslash N_{1,i+1}\,(i\leqslant m_1)$, $R$ --- нормальная подгруппа, порожденная в группе $F$ элементом $r$, $H = \textup{гр }(y_1,\ldots,y_{n-1})$.\\
Если (и только если) элемент $r$ не сопряжен по модулю $N_{1,i+1}$ ни с каким словом от $y_1^{\pm 1},\ldots,y_{n-1}^{\pm 1}$, то
$H\cap RN_{kl} = H\cap N_{kl}$, где $N_{kl}$ --- произвольный член ряда {\rm (\ref{end_gr})}.
\end{theorem}

\begin{corollary} \cite{Km}
Пусть $F$ --- свободная группа с базой $y_1,\ldots,y_n$ $(n>2)$,
\begin{eqnarray}\label{end2_gr}
F=F_{11} \geqslant \ldots \geqslant F_{1,m_1+1}=F_{21} \geqslant \ldots \geqslant F_{s,m_s+1},
\end{eqnarray}
где $F_{ij}$ --- j-й член нижнего центрального ряда группы $F_{i1}$, $s \geqslant 1$.\\
Пусть, далее, $r\in F_{i,j}\backslash F_{i,j+1}\,(j\leqslant m_i)$, $R$ --- нормальная подгруппа, порожденная в группе $F$ элементом $r$, $H = \textup{гр }(y_1,\ldots,y_{n-1})$.\\
Если (и только если) элемент $r$ не сопряжен по модулю $F_{i,j+1}$ ни с каким словом от $y_1^{\pm 1},\ldots,y_{n-1}^{\pm 1}$, то
$H\cap RF_{kl} = H\cap F_{kl}$, где $F_{kl}$ --- произвольный член ряда {\rm (\ref{end2_gr})}.
\end{corollary}

\section{Некоторые свойства производных Фокса для алгебр Ли}
\noindent Все алгебры будут рассматриваться над произвольным фиксированным полем $P$. Пусть $L$ --- алгебра Ли. Через $U(L)$
будем обозначать универсальную обертывающую алгебру алгебры $L$, через $L_{(k)}$ --- $k$-ю степень $L$.\\
Пусть $u$, $v\in L$. Алгебра $L$ вкладывается в $U(L)$ и $[u,v]=uv-vu$ в $U(L)$ (мы обозначаем через $[u,v]$
умножение в $L$). Если $M$ --- идеал в
$L$, то через $M_U$ будем обозначать идеал, порожденный $M$ в $U(L)$. Если $M$ --- идеал алгебры $L$, порожденный
множеством $X$, то будем писать $M=\mbox{ид}_L (X)$.\\
Теорема Пуанкаре-Биркгофа-Витта показывает, что
если $u_1,\ldots,u_n,\ldots$ --- упорядоченный базис в $L$, то $1$ и одночлены вида $u_{i_1}\cdots u_{i_r}$, где $i_1\leqslant\ldots\leqslant i_r$, образуют базис в $U(L)$, который называется стандартным базисом в $U(L)$.\\
Пусть $F$ --- свободная алгебра Ли с базой $\{g_j | j\in J\}$.
Если $u\in U(F)$, то однозначно находятся элементы $D_k(u)\in U(F)$
такие, что
\begin{eqnarray*}
u = \sum_{j\in J} g_jD_j(u).
\end{eqnarray*}
Назовем $D_k(u)$ $(k\in J)$ производными Фокса элемента $u$. Нетрудно видеть, что производные Фокса обладают
следующими свойствами:
\begin{gather}
D_k(\alpha u+\beta v)=\alpha D_k(u)+\beta D_k(v)~(k\in J);\notag\\
D_j(g_j)=1~(j\in J),~D_k(g_j)=0,~\mbox{если}~k\neq j;\notag\\
D_k([u,v])= D_k(u)v-D_k(v)u;~D_k([n,u])\equiv D_k(n)u\mod{N_U}~(k\in J);\notag
\end{gather}
где $u,~v\in F,~n\in N,~\alpha,~\beta\in P$.\\

\begin{lemma}\label{alg1_lm1}
Пусть $F$ --- свободная алгебра Ли с базой $\{g_j | j\in J\}$, $N$ --- идеал в $F$,
$D_j~(j\in J)$ --- производные Фокса алгебры $F$, $K\subseteq J$, $F_K$
 --- подалгебра в $F$, порожденная $\{g_j| j\in K\}$, $\{u_j|j\in K\}$ --- элементы из $U(F_K)$, почти все равные нулю.
Если
\begin{eqnarray}\label{alg1_lm1_1}
\sum_{j\in K} g_ju_j\equiv 0\mod{N_U},
\end{eqnarray}
то найдется $v\in F_K\cap N$ такой, что $D_j(v)\equiv u_j \mod{N_U}$ $(j\in K)$.
\end{lemma}
\begin{proof}
Пусть $\{a_i,~i\in J_1\}$ --- базис пространства $F_K\cap N$.
Дополним его элементами $\{b_j,~j\in J_2\}$ до базиса
$F_K$, потом базис $F_K$ дополним до базиса $F_K+N$ элементами $\{c_s,~s\in
J_3\}$ из $N$.  Базис $F_K+N$ дополним до базиса $F$ элементами $\{d_t,~t\in
J_4\}$. Полагая $a_i<b_j<c_s<d_t$
получаем стандартный базис алгебры $U(F)$, состоящий из слов
\begin{eqnarray}\label{alg1_1}
a_{i_1}\ldots a_{i_\mu}b_{j_1}\ldots b_{j_\nu}c_{s_1}\ldots
c_{s_\eta}d_{t_1}\ldots d_{t_\theta},
\end{eqnarray}
где $i_1\leqslant\ldots\leqslant i_\mu,~j_1\leqslant\ldots\leqslant
j_\nu,~s_1\leqslant\ldots\leqslant s_\eta,~t_1\leqslant\ldots\leqslant
t_\theta,~\mu\geqslant 0,~\nu\geqslant 0,~\eta\geqslant 0,~\theta\geqslant 0$. Из
(\ref{alg1_lm1_1}) следует, что $\sum_{j\in K} g_ju_j$ ---
линейная комбинация одночленов вида (\ref{alg1_1}),
для которых $\mu\geq 1$, $\eta=\theta=0$.
Следовательно,
\begin{eqnarray*}
\sum_{j\in K} g_ju_j = \sum_{x\in
X} n_xw_{x1}\ldots w_{xz_x},
\end{eqnarray*}
где $n_x\in F_K\cap N,~w_{pq}\in F_K,~z_x\geq 0$.\\
Полагаем $v = \sum_{x\in X} [\ldots[n_xw_{x1}]\ldots
w_{xz_x}]$. Тогда $v\in F_K\cap N$ и
\begin{eqnarray}\label{alg1_lm1_2}
D_j(v) \equiv \sum_{x\in X} D_j(n_x)w_{x1}\ldots
w_{xz_x}\mod{N_U}~(j\in K).
\end{eqnarray}
Будем иметь
\begin{gather}
\sum_{j\in K} g_j\sum_{x\in X} D_j(n_x)w_{x1}\ldots w_{xz_x} =\notag\\
\sum_{x\in X}\sum_{j\in K}
g_jD_j(n_x)w_{x1}\ldots w_{xz_x} =\notag\\
\sum_{x\in X}n_xw_{x1}\ldots w_{xz_x}=\sum_{j\in K} g_ju_j.\notag
\end{gather}
Таким образом,
\begin{eqnarray}\label{alg1_lm1_3}
\sum_{x\in X} D_j(n_x)w_{x1}\ldots w_{xz_x}=u_j~(j\in K).
\end{eqnarray}
Из (\ref{alg1_lm1_2}), (\ref{alg1_lm1_3}) следует $D_j(v) \equiv
u_j\mod{N_U}~(j\in K)$.
\end{proof}

\begin{lemma}\label{alg1_tm1}
Пусть $F$ --- свободная алгебра Ли с базой $\{g_j | j\in J\}$, $N$ --- идеал в $F$,
$D_j~(j\in J)$ --- производные Фокса алгебры $F$, $K\subseteq J$, $F_K$
 --- подалгебра в $F$, порожденная $\{g_j| j\in K\}$, $\{u_j|j\in K\}$ --- элементы из $U(F)$, почти все равные нулю.
Если
\begin{eqnarray}\label{alg1_tm1_1}
\sum_{j\in K} g_ju_j\equiv 0\mod{N_U},
\end{eqnarray}
то найдется $v\in \mbox{\rm ид}_F (F_K\cap N)$ такой, что $D_j(v)\equiv u_j \mod{N_U}$ $(j\in K)$.
\end{lemma}
\begin{proof}
Мы можем и будем считать, что стандартный базис алгебры $U(F)$ состоит из слов вида
(\ref{alg1_1}). Найдутся попарно различные стандартные одночлены $f_1,\ldots, f_m$, для которых $\mu+\nu+\eta=0$, такие, что
\begin{eqnarray}\label{alg1_tm1_2}
u_j \equiv \sum_{l=1}^m u_{jl}f_l\mod{N_U}~(j\in K),
\end{eqnarray}
где $u_{jl}\in U(F_K)$. Из (\ref{alg1_tm1_1}) и
\begin{eqnarray*}
\sum_{j\in K} g_ju_j \equiv
\sum_{l=1}^m  \sum_{j\in K}
g_ju_{jl}f_l\mod{N_U},
\end{eqnarray*}
следует
\begin{eqnarray}\label{alg1_tm1_3}
\sum_{j\in K} g_ju_{jl} \equiv
0 \mod{N_U}~(1\leqslant l\leqslant m).
\end{eqnarray}
Ввиду леммы \ref{alg1_lm1} из (\ref{alg1_tm1_3}) следует существование
элементов $v_1,\ldots, v_m$ из $F_K\cap N$ таких, что
\begin{eqnarray}\label{alg1_tm1_4}
D_j(v_l) \equiv u_{jl}\mod{N_U}~(j\in K).
\end{eqnarray}
Пусть $f_l=d_{l1}\ldots d_{lz_l}$. Полагаем $v =
\sum_{l=1}^m [\ldots[v_ld_{l1}]\ldots d_{lz_l}]$.\\
Тогда
$v\in\mbox{id}_F (F_K\cap N)$ и
\begin{eqnarray}\label{alg1_tm1_5}
D_j(v) \equiv \sum_{l=1}^m D_j(v_l)f_l\mod{N_U}~(j\in K).
\end{eqnarray}
Из (\ref{alg1_tm1_4}), (\ref{alg1_tm1_5}) следует
\begin{eqnarray}\label{alg1_tm1_6}
D_j(v) \equiv \sum_{l=1}^m u_{jl}f_l\mod{N_U}~(j\in K).
\end{eqnarray}
Из (\ref{alg1_tm1_2}), (\ref{alg1_tm1_6}) видно, что $D_j(v)
\equiv u_j\mod{N_U}$ $(j\in K)$.
\end{proof}

\begin{theorem}\label{alg1_tm2}
Пусть $F$ --- свободная алгебра Ли с базой $\{g_j | j\in J\}$, $N$ --- идеал в $F$,
$D_j~(j\in J)$ --- производные Фокса алгебры $F$, $K\subseteq J$, $F_K$
 --- подалгебра в $F$, порожденная $\{g_j| j\in K\}$, $v\in F$.
Тогда
\begin{eqnarray}\label{alg1_tm2_1}
D_k(v)\equiv ~0~ mod ~N_U,~k\in J\verb|\|K,
\end{eqnarray}
если и только если найдутся $v_0\in F_K$ и $v_1\in \mbox{\rm ид}_F (F_K\cap N)$ такие, что $v\equiv v_0 + v_1
\mod{[N,N]}$.
\end{theorem}
\begin{proof}Ясно, что если $v\equiv v_0 + v_1 \mod{[N,N]}$, где $v_0\in F_K$ и $v_1\in \mbox{ид}_F (F_K\cap N)$, то
формула (\ref{alg1_tm2_1}) верна.\\
Докажем обратное утверждение теоремы. Мы можем и будем считать, что стандартный базис алгебры $U(F)$ состоит из слов вида
(\ref{alg1_1}). Из (\ref{alg1_tm2_1}) следует, что $v$ --- линейная комбинация стандартных одночленов, для которых
$\mu+\nu+\eta = 1$, откуда $v\in F_K+N$.
Выберем $v_0\in F_K$ так, чтобы было $v-v_0\in N$. Тогда
\begin{eqnarray}\label{alg1_tm2_2}
\sum_{j\in K}
g_jD_j(v-v_0)\equiv 0 \mod{N_U}.
\end{eqnarray}
Ввиду леммы \ref{alg1_tm1} из (\ref{alg1_tm2_2}) следует существование $v_1\in\mbox{id}_F (F_K\cap N)$
такого, что
\begin{eqnarray*}
D_k(v-v_0)\equiv D_k(v_1)\mod{N_U},~k\in J.
\end{eqnarray*}
Тогда
\begin{eqnarray*}
D_k(v-v_0 - v_1)\equiv ~0\mod{N_U},~k\in J,
\end{eqnarray*}
откуда $v-v_0 - v_1\in [N,N]$ \cite{Hm}.
\end{proof}

\section{Теорема о свободе для относительно свободных алгебр Ли с одним определяющим соотношением}
\noindent 
Пусть $F$ --- свободная алгебра Ли с базой $y_1,\ldots,y_n$, $D_j~(j=1,\ldots,n)$ --- производные Фокса алгебры $F$, $H$ --- подалгебра, порожденная в алгебре $F$ элементами $y_1,\ldots,y_{n-1}$, $B$ --- идеал алгебры $F$.
Пусть $a_1,\ldots ,a_r,\ldots$ --- базис пространства $B\cap H$.
Дополним его до базиса $B$ элементами $b_1,\ldots ,b_l,\ldots$.  Базис $B$ дополним до базиса $H+B$ элементами $c_1,\ldots ,c_k,\ldots$ из $H$. Базис $H+B$ дополним до базиса $F$ элементами
$d_1,\ldots ,d_k,\ldots $. Полагая $d_t<c_i<b_j<a_s$
получаем стандартный базис алгебры $U(F)$, состоящий из слов
\begin{eqnarray}\label{alg2_1}
d_{t_1}\ldots
d_{t_\theta}c_{i_1}\ldots
c_{i_\eta}b_{j_1}\ldots b_{j_\nu}a_{s_1}\ldots a_{s_\mu},
\end{eqnarray}
где $t_1\leqslant\ldots\leqslant t_\theta,~i_1\leqslant\ldots\leqslant i_\eta,
~j_1\leqslant\ldots\leqslant j_\nu,~s_1\leqslant\ldots\leqslant s_\mu$,\\
$\theta\geqslant 0,~\eta\geqslant 0,~\nu\geqslant 0,~\mu\geqslant 0$.\\
Упорядочим стандартный базис алгебры $U(F)$ при помощи отношения $\leqslant$ так, что одночлены меньшей длины предшествуют
одночленам большей длины, а одночлены равной длины упорядочены лексикографически (слева направо).
Если $u\in U(F)$, то через $u\rightarrow \bar{u}$ обозначим функцию, выбирающую из стандартных одночленов,
входящих в разложение $u$ по базису $U(F)$, максимальный.\\
Пусть $S_\alpha$ --- множество стандартных одночленов, для которых $\theta + \nu + \mu = 0$,\\
$S_\beta$ --- множество стандартных одночленов, для которых $\theta> 0$ и $\nu + \mu = 0$.
Будем называть $S=S_\alpha\cup S_\beta$ системой представителей алгебры $U(F)$ по идеалу $B_U$.
Алгебра $B$ --- подалгебра свободной лиевой алгебры, следовательно $B$ --- свободная лиева алгебра \cite{Sh2}.
Пусть $x$ --- элемент базы алгебры $B$.  Так как $x\not\in [B,B]$, то найдется производная $D_j$ такая, что
\begin{eqnarray*}
D_j (x)\not\equiv ~0\mod{B_U}.
\end{eqnarray*}
1) Если $D_n (x)\not\equiv ~0\mod{B_U}$, то ставим в соответствие элементу $x$ строку $(M(x),n)$, где
$M(x)$ какой-нибудь элемент из $S$, входящий в разложение $D_n (x)$ по базису $U(F)$.\\
2) Если $D_n (x)\equiv ~0\mod{B_U}$ и найдутся $D_j (x)$,
$M(x)\in S_\beta$ такие, что $M(x)$ входит в разложение $D_j (x)$ по базису $U(F)$,
то ставим в соответствие элементу $x$ строку $(M(x),j)$.\\
3) Если $D_n (x)\equiv ~0\mod{B_U}$ и в разложения элементов $D_j (x)\,(j\neq n)$ по базису $U(F)$
не входят стандартные одночлены из $S_\beta$, то ставим в соответствие элементу $x$ строку $(M(x),j)$, где $M(x)$ какой-нибудь элемент из $S_\alpha$, входящий в разложение $D_j (x)$ по базису $U(F)$.\\
Если для элемента $x$ базы алгебры $B$ выполнено условие 3), то по лемме~\ref{alg1_lm1}
найдется $y\in H\cap B$ такой, что $D_j (x)\equiv D_j (y)\mod{B_U}\,(j=1,\ldots,n)$, следовательно, $x\in H\cap B\mod{[B,B]}$ \cite{Hm}.\\
Пусть $r\in B$, $Z$ --- множество свободных порождающих алгебры $B$, $z_1,\ldots,z_p$ --- попарно различные элементы из $Z$, такие, что $r$ принадлежит алгебре, порожденной этими элементами. Пусть $(M(z_1),j_1)$ --- строка, поставленная в соответствие элементу $z_1$ описанным выше способом.
Если в $z_2,\ldots,z_p$ найдется элемент $z_i$, которому может быть поставлена в соответствие строка равная $(M(z_1),j_1)$, то
выберем $\gamma\in P$ так, чтобы $M(z_1)$ не входил в разложение
$D_{j_1} (z_i-\gamma z_1)$ по базису $U(F)$. Заменим $z_i$ на $z_i-\gamma z_1$. Продолжая аналогичные рассуждения, найдем такое множество $X$ свободных порождающих алгебры $B$ и такие попарно различные элементы $x_1,\ldots,x_p$ из $X$, что
$r$ принадлежит алгебре, порожденной $x_1,\ldots,x_p$, каждому $x_i$ поставлена в соответствие строка $(M(x_i),j_i)$ описанным выше способом такая, что $M(x_i)$ входит в разложение
$D_{j_i} (x_i)$ по базису $U(F)$ и $M(x_i)$ не входит в разложение
$D_{j_i} (x_l)$ по базису $U(F)$, $l\neq i$. Обозначим через $\{\partial_z| z\in {\bf N}\}$ производные Фокса алгебры $B$ соответствующие базе $X$.\\
Так как $D_{j_i}(r)= \sum_{l={1}}^p D_{j_i}(x_l) \partial_l(r)$, то
\begin{eqnarray}\label{main}
D_{j_i}(r)= M(x_i)\partial_i(r)+\sum_{c={1}}^g t_c \lambda_c,
\end{eqnarray}
где $\lambda_c\in \{\partial_1(r),\ldots,\partial_p(r)\}$, $t_c$ --- стандартный одночлен,
не равный $M(x_i)$.\\
Пусть $v\in B\cap H$. Аналогично тому, как строилась база $X$, построим такую базу $\{\bar{x}_z | z\in {\bf N}\}$
алгебры $B\cap H$, что $v$ принадлежит алгебре, порожденной $\bar{x}_1,\ldots,\bar{x}_{\bar{p}}$,
элементу $\bar{x}_z$ поставлена в соответствие строка $(M(\bar{x}_z),\bar{j}_z)$,
$M(\bar{x}_z)\in S_\alpha$, $\bar{j}_z\in \{1,\ldots,n-1\}$ $(z=1,\ldots,\bar{p})$,
$M(\bar{x}_z)$ входит в разложение $D_{\bar{j}_z} (\bar{x}_z)$ по базису $U(H)$
и $M(\bar{x}_z)$ не входит в разложение $D_{\bar{j}_z} (\bar{x}_l)$ по базису $U(H)$, $l\neq z$.\\
Обозначим через $\{\bar{\partial}_z| z\in {\bf N}\}$ соответствующие этой базе производные Фокса алгебры $B\cap H$.
Тогда
\begin{eqnarray}\label{main2}
D_{\bar{j}_i}(v)= M(\bar{x}_i)\bar{\partial}_i(v)+\sum_{c={1}}^{\bar{g}} \bar{t}_c \bar{\lambda}_c,
\end{eqnarray}
где $\bar{\lambda}_c\in \{\bar{\partial}_1(v),\ldots,\bar{\partial}_{\bar{p}}(v)\}$, $\bar{t}_c$ --- стандартный одночлен,
не равный $M(\bar{x}_i)$.\\

\begin{lemma}\label{lm1_2}
Пусть $F$ --- свободная алгебра Ли с базой $y_1,\ldots,y_n$, $D_j~(j=1,\ldots,n)$ --- производные Фокса алгебры $F$, $H$ --- подалгебра, порожденная в алгебре $F$ элементами $y_1,\ldots,y_{n-1}$, $B$ --- идеал алгебры $F$,
$S=S_\alpha\cup S_\beta$ --- система представителей алгебры $U(F)$ по идеалу $B_U$, $u\rightarrow \bar{u}$ --- выбирающая функция, $\delta_1,\ldots,\delta_l$, $\mu_1,\ldots,\mu_k$ --- элементы из $S$,
$\delta_i <\delta_j$, $\mu_i <\mu_j$ при $i <j$.\\
Тогда если $\{\delta_1,\ldots,\delta_l,\mu_1,\ldots,\mu_k\}\not\subseteq S_\alpha$, то
найдутся $\mu\in S_\beta,\,i_0,\,j_0$ такие, что $\mu$ входит в разложение $\delta_{i_0}\mu_{j_0}$ по базису $U(F)$
и не входит в разложение $\delta_i\mu_j$ по базису $U(F)$ при $(i_0,j_0)\neq (i,j)$.
\end{lemma}
\begin{proof}
Если $u$ --- стандартный одночлен, то полагаем $d(u)=\theta$.
Пусть $\{\delta_1,\ldots,\delta_l,\mu_1,\ldots,\mu_k\}\not\subseteq S_\alpha$.
Обозначим через $x$ --- максимальный элемент из $\{d(\delta_1),\ldots,d(\delta_l)\}$, через $y$ --- максимальный элемент из $\{d(\mu_1),\ldots,d(\mu_l)\}$. Тогда $x+y>0$. Пусть $\delta_{i_0}$ --- максимальный элемент из тех $\delta_i$, для которых
$d(\delta_i)=x$, $\mu_{j_0}$ --- максимальный элемент из тех $\mu_j$, для которых
$d(\mu_j)=y$. Так как $\mu=\overline{\delta_{i_0}\mu_{j_0}}$ --- максимальный из стандартных одночленов,
входящих в разложение $\delta_{i_0}\mu_{j_0}$ по базису $U(F)$, то $\mu\in S_\beta$
и $\mu$ не входит в разложение $\delta_i\mu_j$ по базису $U(F)$ при $(i_0,j_0)\neq (i,j)$.
\end{proof}

\begin{lemma}\label{lm2_2}
Пусть $F$ --- свободная алгебра Ли с базой $\{g_j | j\in J\}$,
$U(F)$ --- универсальная обертывающая алгебра, $U_0(F)$ --- идеал, порожденный $F$ в $U(F)$, $v\in F$. Тогда и только тогда $v\in F_{(n)}\setminus F_{(n+1)}$, когда
$v\in U_0(F)^n\setminus  U_0(F)^{n+1}$.
\end{lemma}
\begin{proof}
Известно, что $U(F)$ --- свободная ассоциативная алгебра с единицей и с базой $\{g_j | j\in J\}$, откуда следует утверждение леммы.
\end{proof}


\begin{lemma}\label{lm5_2}\cite{Um}
Пусть $F$ --- свободная алгебра Ли с базой $\{g_j | j\in J\}$,
$D_j~(j\in J)$ --- соответствующие этой базе производные Фокса алгебры $F$.
Пусть, далее, $H$ --- подалгебра алгебры $F$, $f\in H$,
$\{h_i | i\in I\}$ --- база $H$, $\{\partial_i | i\in I\}$ --- соответствующие этой базе производные Фокса алгебры $H$.\\
Тогда $D_j(f)=\sum_k D_j(h_k)\partial_k(f)$.
\end{lemma}

\begin{lemma}\label{lm3_2}
Пусть $F$ --- свободная алгебра Ли с базой $\{g_j | j\in {\bf N}\}$, $U(F)$ --- универсальная обертывающая алгебра, $U_0(F)$ --- идеал, порожденный $F$ в $U(F)$, $(F_{(n)})_U$ --- идеал, порожденный $F_{(n)}$ в $U(F)$, $v\in F$, $D_j~(j\in {\bf N})$ --- производные Фокса алгебры $F$.
Если $D_1(v)\notin U_0(F)^{k-1}$, то $D_1([v,g_2])\notin U_0(F)^k$. 
\end{lemma}
\begin{proof}
Ясно, что $D_1([v,g_2])=D_1(v)g_2$.
Так как $U(F)$ --- свободная ассоциативная алгебра с единицей и с базой $\{g_j | j\in J\}$, то, $D_1([v,g_2])\notin U_0(F)^k$.
\end{proof}

\noindent Пусть $G$ --- свободная алгебра Ли, $H$ --- подалгебра алгебры $G$, ряд
\begin{eqnarray*}
G=G_1\geqslant\ldots \geqslant G_l\geqslant\ldots
\end{eqnarray*}
идеалов алгебры $G$ такой, что $[G_p,G_m\,]\leqslant G_{p+m}$.\\
Положим $H_t=H\cap G_t$, $\Delta_i$ --- идеал в $U(G)$, порожденный $G_{i_1}\cdots G_{i_s}$,
$\Delta_i^\prime$ --- идеал в $U(H)$, порожденный $H_{i_1}\cdots H_{i_s}$, где $i_1 + \cdots + i_s\geqslant i$.\\
Нетрудно видеть, что
\begin{gather}
U(H)\cap \Delta_i\subseteq \Delta_i^\prime \mod{(G_m)_U};\notag\\
\text{если } u\in \Delta_i\setminus \Delta_{i+1} \mod{(G_m)_U},~v\in \Delta_j\setminus \Delta_{j+1} \mod{(G_m)_U}, \text{ то }\notag\\
uv\in \Delta_{i+j}\setminus \Delta_{i+j+1} \mod{(G_m)_U}.\notag
\end{gather}

\begin{proposition}\label{tm4}
Пусть $F$ --- свободная алгебра Ли с базой $y_1,\ldots,y_n$, $n\geqslant 3$, $N$ --- эндоморфно допустимый идеал алгебры $F$,
\begin{eqnarray}\label{tm4_0}
N=N_{11} \geqslant \ldots \geqslant N_{1,m_1+1}=N_{21} \geqslant \ldots \geqslant N_{s,m_s+1},
\end{eqnarray}
где $N_{kl}$ --- l-я степень алгебры $N_{k1}$.\\
Пусть, далее, $R$ --- идеал алгебры $F$, $R\leqslant N$, $H$ подалгебра в $F$, порожденная $y_1,\ldots,y_{n-1}$.
Если $H\cap (R+N_{1j})\neq H\cap N_{1j}$, то $H\cap (R+N_{kl})\neq H\cap N_{kl}\, (N_{kl}\leqslant N_{1j})$.
\end{proposition}
\begin{proof}
Отметим, что $H\cap N_{kl}=(H\cap N_{k1})_{(l)}$.\\
Предположим, $H\cap (R+N_{1j})\neq H\cap N_{1j}$. Покажем, что
\begin{eqnarray}\label{pr1}
H\cap (R+N_{1l})> (H\cap N)_{(l)}\,~(N_{1l}\leqslant N_{1j}).
\end{eqnarray}
По условию, $H\cap (R+N_{1j})> (H\cap N)_{(j)}$.
Пусть для всех членов ряда~(\ref{tm4_0}) от $N_{1j}$ до $N_{1l}$ включительно формула~(\ref{pr1}) верна ($l\geqslant j$).
Обозначим через $B$ алгебру $H\cap N$ и через $U_0(B)$ --- идеал, порожденный $B$ в $U(B)$.
Выберем в алгебре $B$ базу $\{x_z | z\in {\bf N}\}$, $\{\partial_z | z\in {\bf N}\}$ --- соответствующие этой базе
производные Фокса алгебры $B$.
Пусть $v\in (H\cap (R+N_{1l}))\setminus B_{(l)}$.
Мы можем и будем считать, что
$\partial_1(v)\notin U_0(B)^{l-1}$.
Полагаем $w=[v,x_2]$.
По лемме~\ref{lm3_2}, $\partial_1(w)\notin U_0(B)^l$.\\
Следовательно, $w\in (H\cap (R+N_{1,l+1}))\setminus B_{(l+1)}$.\\
Теперь соображения индукции заканчивают доказательство формулы~(\ref{pr1}).
Из (\ref{pr1}) следует $H\cap (R+N_{21})> H\cap N_{21}$.\\
Остается заметить, что из $H\cap (R+N_{k1})> H\cap N_{k1}$ следует\\
$(H\cap (R+N_{k1}))_{(l)}> (H\cap N_{k1})_{(l)}~(l\in {\bf N})$, т.е.
$H\cap (R+N_{kl})> H\cap N_{kl}\, (N_{kl}\leqslant N_{21})$.
\end{proof}

\begin{proposition}\label{tm2}
Пусть $F$ --- свободная алгебра Ли с базой $y_1,\ldots,y_n$,
$H$ --- подалгебра, порожденная в алгебре $F$ элементами $y_1,\ldots,y_{n-1}$, $N$ --- эндоморфно допустимый идеал алгебры $F$,
\begin{eqnarray}\label{tm2_0}
N=N_{11} \geqslant \ldots \geqslant N_{1,m_1+1}=N_{21} \geqslant \ldots \geqslant N_{s,m_s+1},
\end{eqnarray}
где $N_{kl}$ --- l-я степень алгебры $N_{k1}$.\\
Пусть, далее, $r\in N_{1i}\backslash N_{1,i+1}\,(i\leqslant m_1)$, $R = \mbox{\rm ид}_F (r)$.
Если $H\cap (R+N_{21})=H\cap N_{21}$, то $H\cap (R+N_{kl})=H\cap N_{kl}\,(k> 1)$.
\end{proposition}
\begin{proof}
\noindent Отметим, что $H\cap N_{kl}=(H\cap N_{k1})_{(l)}$.
Пусть $D_1,\ldots,D_n$ --- производные Фокса алгебры $F$. Докажем, что $D_n(r)\not\equiv 0\mod {(R+N_{21})_U}$. Предположим противное.
Тогда по теореме~\ref{alg1_tm2} нашлись бы $v_1,\ldots,v_d$ из $H\cap (R+N_{21})$; $f_1,\ldots,f_d$ из $F$
такие, что
\begin{eqnarray*}
r\equiv [v_1,f_1]+\ldots +[v_d,f_d]\mod [R+N_{21},R+N_{21}].
\end{eqnarray*}
По условию $H\cap (R+N_{21})=H\cap N_{21}$, поэтому $v_1,\ldots,v_d$ принадлежат $N_{21}$. Следовательно, $r\in N_{1,i+1}$.
Так как $r\in N_{1i}\backslash N_{1,i+1}$, то мы получаем противоречие.\\
Пусть $H\cap (R+N_{kl})=H\cap N_{kl}$ для всех членов ряда (\ref{tm2_0}) от $N_{21}$ до $N_{kl}$
включительно $(N_{21} \geqslant N_{kl})$. Требуется доказать\\
а) $H\cap (R+N_{{k+1},2})=H\cap N_{{k+1},2}\,(l= m_k+1,\,k\neq s)$,\\
б) $H\cap (R+N_{k,{l+1}})=H\cap N_{k,{l+1}}\,(l\leqslant m_k)$.\\
Докажем а). Будем иметь $H\cap (R+N_{k+1,1})=H\cap N_{k+1,1}$.\\
Выберем $v\in H\cap (R+N_{{k+1},2})$. Найдется элемент $\alpha\in U(F)$ такой, что
\begin{eqnarray}\label{tm2_1}
D_m(v)\equiv D_m(r)\cdot\alpha\mod{(R+N_{{k+1},1})_U},~m\in \{1,\ldots,n\}.
\end{eqnarray}
Так как в $U(F/(R+N_{{k+1},1}))=U(F)/(R+N_{{k+1},1})_U$ нет делителей нуля, то из $D_n(r)\not\equiv 0\mod {(R+N_{21})_U}$, $D_n(v)=0$ и (\ref{tm2_1})
следует, что
\begin{eqnarray}\label{tm2_1_1}
\alpha\equiv 0\mod (R+N_{{k+1},1})_U\,.
\end{eqnarray}
Формулы (\ref{tm2_1}), (\ref{tm2_1_1}) показывают, что $D_m(v)\equiv 0\mod (N_{k+1,1})_U,~m\in \{1,\ldots,n\}$, поэтому
$v\in H\cap N_{{k+1},2}$.\\
Докажем б). Пусть $v\in H\cap (R+N_{k,l+1})$. Обозначим алгебру $R+N_{k1}$ через $B$.
Мы можем и будем считать, что стандартный базис алгебры $U(F)$ состоит из слов вида
(\ref{alg2_1}). Пусть $S$ --- система представителей алгебры $U(F)$ по идеалу $B_U$.
Обозначим через $U_0(B)$ идеал, порожденный $B$ в $U(B)$.\\
Выберем в $B$ базу $\{x_z | z\in {\bf N}\}$.
Обозначим через $\{\partial_z | z\in {\bf N}\}$ соответствующие этой базе производные Фокса алгебры $B$.\\
Найдутся $u\in B_{(l+1)}$, $k_i\in U(B)$, $i\in \{1,\ldots,d\}$, $\mu_i\in S~(\mu_i\neq \mu_j\mbox{ при }i\neq j)$ такие, что
\begin{eqnarray}\label{tm2_3}
D_m(v)\equiv D_m(r)\cdot  \sum_{i={1}}^{d} \mu_i k_i +\sum_{z\in {\bf N}} D_m(x_z)\partial_z(u)\mod{(R+N_{kl})_U},
\end{eqnarray}
$m\in \{1,\ldots,n\}$. Будем иметь
\begin{eqnarray}\label{tm2_4}
0\equiv D_n(r)\cdot  \sum_{i={1}}^{d} \mu_i k_i+ \sum_{z\in {\bf N}} D_n(x_z)\partial_z(u)\mod{(R+N_{kl})_U}.
\end{eqnarray}
Из $D_n(r)\not\equiv 0\mod {B_U}$ следует, что
\begin{eqnarray*}
D_n(r) = \sum_{i={1}}^q  \delta_i t_i ,
\end{eqnarray*}
где $t_i\in U(B)$, $\delta_i\in S~(\delta_i\neq \delta_j\mbox{ при }i\neq j)$, $i\in \{1,\ldots,q\}$ и
не все элементы из $\{t_1,\ldots,t_q\}$ принадлежат $U_0(B)$.
Пусть $t_{j_1},\ldots,t_{j_b}$ не принадлежат $U_0(B)$, $\delta_{j_1}<\ldots <\delta_{j_b}$.
Предположим, что не все $k_i$ принадлежат $U_0(B)^l\mod{(R+N_{kl})_U}$.
Выберем минимальное $l_0$ такое, что $k_i\in U_0(B)^{l_0}\mod{(R+N_{kl})_U}$, $i=1,\ldots,d$
и пусть $k_{i_1},\ldots,k_{i_a}$ не принадлежат $U_0(B)^{l_0 +1}\mod{(R+N_{kl})_U}$,
$\mu_{i_1}<\ldots <\mu_{i_a}$.\\
Ясно, что $\overline{\delta_{j_b}\mu_{i_a}}\in S$ и $\overline{\delta_{j_b}\mu_{i_a}}$ больше любого  из стандартных одночленов, входящих в разложение $\delta_{j_c}\mu_{i_d}$ по базису $U(F)$ $((c,d)\neq (b,a))$. Тогда
\begin{eqnarray}\label{tm2_4_1}
D_n(r)\cdot  \sum_{i={1}}^{d} \mu_i k_i = \overline{\delta_{j_b}\mu_{i_a}}t_{ba} +  \sum_{i={1}}^{\hat{q}}  \hat{\delta_i} \hat{t_i} ,
\end{eqnarray}
где $t_{ba}\not\in U_0(B)^{l_0 +1}\mod{(R+N_{kl})_U}$, $\hat{t}_i\in U(B)$, $\hat{\delta}_i\in S, \hat{\delta}_i\neq \overline{\delta_{j_b}\mu_{i_a}}$, $i\in \{1,\ldots,\hat{q}\}$.
По лемме \ref{lm2_2}, все $\partial_z(u)$ лежат в $U_0(B)^l$,
поэтому (\ref{tm2_4_1}) противоречит (\ref{tm2_4}).\\
Таким образом, $k_i\in U_0(B)^l\mod{(R+N_{kl})_U}$, $i\in \{1,\ldots,d\}$.
Следовательно, (\ref{tm2_3}) можно переписать в виде
\begin{eqnarray}\label{tm2_6}
D_m(v)\equiv \sum_{i={1}}^{d_m} \mu_{i,m} g_{i,m} \mod{(R+N_{kl})_U},~m\in \{1,\ldots,n\},
\end{eqnarray}
где $\mu_{i,m}\in S$, $g_{i,m}\in U_0(B)^l$.\\
Пусть $H_t=H\cap (R+N_{kt})$, $1\leqslant t\leqslant l$, $\Delta_l$ --- идеал в $U(H_1)$,
порожденный $H_{i_1}\cdots H_{i_s}$, $i_1 + \cdots + i_s\geqslant l$, $\mathfrak{X}$ --- идеал, порожденный $H_1$ в $U(H_1)$,
база $\{\bar{x}_z | z\in {\bf N}\}$ в алгебре $H_1$ выбрана так, что верна формула (\ref{main2}).\\
Из $v\in H\cap (R+N_{k,l+1})$ следует $v\in H\cap N_{kl}=(H\cap N_{k1})_{(l)}$, поэтому
$\bar{\partial}_z(v)\in \mathfrak{X}^{l-1}$ и формулы
(\ref{main2}), (\ref{tm2_6}) показывают, что
$\bar{\partial}_z(v)\in \Delta_l\mod{(R+N_{kl})_U},~z\in {\bf N}$.\\
По условию, $H_t=H\cap N_{kt}=(H\cap N_{k1})_{(t)} \subseteq \mathfrak{X}^t$, следовательно
\begin{eqnarray*}
\bar{\partial}_z(v)\in \mathfrak{X}^l\mod{(R+N_{kl})_U},~z\in {\bf N}.
\end{eqnarray*}\\
Так как $H\cap (R+N_{kl})=H\cap N_{kl}=(H\cap N_{k1})_{(l)} $, то $\bar{\partial}_z(v)\in \mathfrak{X}^l,~z\in {\bf N}$.
Это означает, что $v\in \mathfrak{X}^{l+1}$, следовательно $v\in (H\cap N_{k1})_{(l+1)}$.
Теперь соображения индукции заканчивают доказательство.
\end{proof}

\begin{proposition}\label{tm5}
Пусть $F$ --- свободная алгебра Ли с базой $y_1,\ldots,y_n$,
$H$ --- подалгебра, порожденная в алгебре $F$ элементами $y_1,\ldots,y_{n-1}$, $N$ --- эндоморфно допустимый идеал алгебры $F$,
\begin{eqnarray}\label{tm5_0}
N=N_{11} \geqslant \ldots \geqslant N_{1,m_1+1}=N_{21} \geqslant \ldots \geqslant N_{s,m_s+1},
\end{eqnarray}
где $N_{kl}$ --- l-я степень алгебры $N_{k1}$.\\
Пусть, далее, $r\in N_{1i}\backslash N_{1,i+1}\,(i\leqslant m_1)$, $R = \mbox{\rm ид}_F (r)$.
Если $r\not\in H+N_{1,i+1}$, то $H\cap (R+N_{1l})=H\cap N_{1l}\,(l\leqslant m_1+1)$.
\end{proposition}
\begin{proof}
\noindent Отметим, что $H\cap N_{1l}=(H\cap N_{11})_{(l)}$.\\
Ясно, что $H\cap (R+N_{1i})=H\cap N_{1i}$.
Требуется доказать, что
\begin{eqnarray}\label{tm3_0_0}
H\cap (R+N_{1l})=H\cap N_{1l}\,(l\geqslant i).
\end{eqnarray}
Пусть для всех членов ряда (\ref{tm5_0}) от $N_{1i}$ до $N_{1l}$
включительно $(N_{1i} \geqslant N_{1l})$ формула (\ref{tm3_0_0}) верна.
Покажем, что $H\cap (R+N_{1,{l+1}})=H\cap N_{1,{l+1}}\,(l\leqslant m_1)$.\\
Пусть $D_1,\ldots,D_n$ --- производные Фокса алгебры $F$, $v\in H\cap (R+N_{1,l+1})$. Обозначим алгебру $N_{11}$ через $B$.
Мы можем и будем считать, что стандартный базис алгебры $U(F)$ состоит из слов вида
(\ref{alg2_1}).\\
Обозначим через $S$ множество стандартных одночленов, для которых $\mu = 0$, через $U_0(B)$ --- идеал, порожденный
$B$ в $U(B)$.\\
Пусть база $\{x_z | z\in {\bf N}\}$ в алгебре $B$ выбрана так, что $r$ принадлежит алгебре, порожденной
$x_1,\ldots,x_p$ и верна формула (\ref{main}).
Найдутся $u\in B_{(l+1)}$, $k_i\in U(B)$, $i\in \{1,\ldots,d\}$, $\mu_i\in S~(\mu_i\neq \mu_j\mbox{ при }i\neq j)$ такие, что
\begin{eqnarray}\label{tm3_3}
D_m(v)\equiv D_m(r)\cdot  \sum_{i={1}}^{d} \mu_i k_i +\sum_{z\in {\bf N}} D_m(x_z)\partial_z(u)\mod{(R+N_{1l})_U},
\end{eqnarray}
$m\in \{1,\ldots,n\}$. Будем иметь
\begin{eqnarray}\label{tm3_4}
0\equiv D_n(r)\cdot  \sum_{i={1}}^{d} \mu_i k_i+ \sum_{z\in {\bf N}} D_n(x_z)\partial_z(u)\mod{(R+N_{1l})_U}.
\end{eqnarray}
Обозначим через $\Delta_l$ --- идеал
в $U(B)$, порожденный $(R+N_{1i_1})\cdots (R+N_{1i_s})$, $i_1 + \cdots + i_s\geqslant l$.
Полагаем $U_0(B)^0=\Delta_0=U(B)$.
Отметим, что если $l\leqslant i$, то $U_0(B)^l=\Delta_l$.\\
По лемме \ref{lm2_2}, $\partial_z(u)\in U_0(B)^l$, $\partial_z(v)\in U_0(B)^{l-1}$, $\partial_z(r)\in U_0(B)^{i-1}$ $(z\in {\bf N})$.\\
Покажем, что
\begin{eqnarray}\label{main3}
\{D_1(v),\ldots,D_{n-1}(v)\} \subseteq U(F)\Delta_l.
\end{eqnarray}
Предположим, формула (\ref{main3}) неверна.
Так как $r\in N_{1i}\setminus  N_{1,i+1}$, то найдется $t$ такое, что $\partial_t(r) \in \Delta_{i-1}\setminus \Delta_i$.
Формула (\ref{main}) показывает, что
\begin{eqnarray}\label{tm3_4_0}
D_{j_t}(r) = \delta_1 a_1+ \cdots +\delta_q a_q,
\end{eqnarray}
$\delta_k\in S$, $\delta_k\neq \delta_p$ при $k\neq p$, $a_p \in \Delta_{i-1}$ и $a_p \in \Delta_{i-1}\setminus \Delta_i$ для некоторых $p$.
В рассматриваемом случае из (\ref{tm3_3}), (\ref{tm3_4_0}) следует $k_p \in \Delta_{l-i}$, $p\in \{1,\ldots,d\}$ и
$k_p \in \Delta_{l-i}\setminus \Delta_{l-i+1}$ для некоторых $p$.
Пусть $M$ --- подмножество в $\{1,\ldots,d\}$ такое, что $k_p \in \Delta_{l-i}\setminus \Delta_{l-i+1}$ при $p\in M$,
$N$ --- подмножество в $\{1,\ldots,q\}$ такое, что $a_p \in \Delta_{i-1}\setminus \Delta_i$ при $p\in N$.\\
Так как $D_m(v)$ --- сумма элементов вида $\theta b$, где $\theta\in S_\alpha$, $b\in U(B)\cap U(H)$,
то (\ref{tm3_3}) показывает, что не существует таких $\mu\in S_\beta$, $k_0, k \in M$, $p_0, p \in N$, $(k_0,p_0)\neq (k,p)$, что $\mu$ входит в разложение $\delta_{k_0}\mu_{p_0}$ по базису $U(F)$ и не входит в разложение $\delta_k\mu_p$ по базису $U(F)$.\\
Тогда, ввиду леммы \ref{lm1_2}, если $k\in M$, то $\mu_k\in S_\alpha$, т.е.
\begin{eqnarray}\label{tm3_4_1_1}
\sum_{i={1}}^{d} \mu_i k_i\equiv \hat{\mu}_1 b_1+ \cdots +\hat{\mu}_{\hat{d}} b_{\hat{d}}\mod{U(F)\Delta_{l-i+1}},
\end{eqnarray}
$\hat{\mu}_k\in S_\alpha$, $\hat{\mu}_k\neq \hat{\mu}_p$ при $k\neq p$, $b_k\in \Delta_{l-i}\setminus \Delta_{l-i+1}$.\\
Пусть $x_t\in \{x_1,\ldots,x_p\}$. \\
Рассмотрим сначала случай, когда элементу $x_t$ поставлена в соответствие строка $(M(x_t),n)$ описанным выше способом 1).\\
Если $\partial_t(r)\in\Delta_{i-1}\setminus \Delta_i$, то формулы (\ref{main}), (\ref{tm3_4_1_1}) показывают, что
\begin{eqnarray}\label{tm3_4_2}
D_n(r)\cdot  \sum_{i={1}}^{d} \mu_i k_i\equiv \hat{\delta}_1\hat{a}_1+ \cdots +\hat{\delta}_s\hat{a}_s\mod{U(F)\Delta_l},
\end{eqnarray}
$\hat{\delta}_k\in S$, $\hat{\delta}_k\neq \hat{\delta}_p$ при $k\neq p$, $\{\hat{a}_1,\ldots,\hat{a}_s\}\subset U(B)$ и $\hat{a}_p \in \Delta_{l-1}\setminus \Delta_l$ для некоторых $p$.
Формула (\ref{tm3_4_2}) противоречит (\ref{tm3_4}), следовательно $\partial_t(r)\in \Delta_i$.\\
Рассмотрим теперь случай, когда элементу $x_t$ поставлена в соответствие строка $(M(x_t),j_t)$, $M(x_t)\in S_\beta$,
$j_t\neq n$ описанным выше способом 2).\\ 
Если $\partial_t(r)\in\Delta_{i-1}\setminus \Delta_i$, то формулы (\ref{main}), (\ref{tm3_4_1_1}) показывают, что
\begin{eqnarray}\label{tm3_4_2_2}
D_{j_t}(r)\cdot  \sum_{i={1}}^{d} \mu_i k_i\equiv \hat{\delta}_1\hat{a}_1+ \cdots +\hat{\delta}_s\hat{a}_s\mod{U(F)\Delta_l},
\end{eqnarray}
$\hat{\delta}_k\in S$, $\hat{\delta}_k\neq \hat{\delta}_p$ при $k\neq p$, $\{\hat{a}_1,\ldots,\hat{a}_s\}\subset U(B)$ и, для некоторых $p$, $\hat{\delta}_p\in S_\beta$, $\hat{a}_p \in \Delta_{l-1}\setminus \Delta_l$.
Так как $D_{j_t}(v)$ --- сумма элементов вида $\theta b$, где $\theta\in S_\alpha$, $b\in U(B)\cap U(H)$,
то формула (\ref{tm3_4_2_2}) противоречит (\ref{tm3_3}), следовательно $\partial_t(r)\in \Delta_i$.\\
Рассмотрим, наконец, случай, когда элементу $x_t$ поставлена в соответствие строка $(M(x_t),j_t)$, $M(x_t)\in S_\alpha$,
$j_t\neq n$ описанным выше способом 3).\\
Тогда, как было показано выше, $x_t\in H\cap B\mod{[B,B]}$.\\
Если $\partial_t(r)\not\in U(H)\mod{\Delta_i}$, то, ввиду (\ref{tm3_4_1_1}),
\begin{eqnarray}\label{tm3_4_2_2_0}
D_{j_t}(r)\cdot  \sum_{i={1}}^{d} \mu_i k_i\equiv \hat{\delta}_1\hat{a}_1+ \cdots +\hat{\delta}_s\hat{a}_s\mod{U(F)\Delta_l},
\end{eqnarray}
$\hat{\delta}_k\in S_\alpha$, $\hat{\delta}_k\neq \hat{\delta}_p$ при $k\neq p$, $\{\hat{a}_1,\ldots,\hat{a}_s\}\subset\Delta_{l-1}$,
$\{\hat{a}_1,\ldots,\hat{a}_s\}\not\subset U(H)+\Delta_l$. Так как $D_{j_t}(v)$ --- сумма элементов вида $\theta b$, где $\theta\in S_\alpha$, $b\in U(B)\cap U(H)$,
то формула (\ref{tm3_4_2_2_0}) противоречит (\ref{tm3_3}), следовательно $\partial_t(r)\in U(H)\mod{\Delta_i}$.\\
Пусть $\psi\text{: }U(F)\rightarrow U(F)$ --- эндоморфизм, определяемый отображением
$y_j\rightarrow y_j$ при $j\in \{1,\ldots,n-1\}$, $y_n\rightarrow 0$. Обозначим $\psi(r)$ через $\hat{r}$.\\
Тогда $\hat{r}= \sum_{t={1}}^{p } \psi(x_t)\psi(\partial_t(r))$, где либо $\psi(\partial_t(r))\in \Delta_i$,
либо верны следующие соотношения
\begin{gather}
\psi(x_t)\equiv x_t\mod{[B,B]};\notag\\
\psi(\partial_t(r))\equiv  \partial_t(r)\mod{\Delta_i}.\notag
\end{gather}
Следовательно, $r -\hat{r}\in U_0(B)^{i+1}$, т.е. $r\in H+N_{1,i+1}$. Из полученного противоречия следует (\ref{main3}).\\
Пусть $H_t=H\cap (R+N_{1t})$, $1\leqslant t\leqslant l$, $\Delta_l^\prime$ --- идеал в $U(H_1)$, порожденный $H_{i_1}\cdots H_{i_s}$, $i_1 + \cdots + i_s\geqslant l$ и база $\{\bar{x}_z | z\in {\bf N}\}$ в алгебре $B\cap H$ выбрана так, что верна формула (\ref{main2}).
Из (\ref{main2}), (\ref{main3}) следует, что $\bar{\partial}_z(v)\in \Delta_l^\prime\mod{(R+N_{1l})_U},~z\in {\bf N}$.\\
Обозначим через $\mathfrak{X}$ идеал, порожденный $H\cap N_{11}$ в $U(H\cap N_{11})$.
По условию, $H_t=(H\cap N_{11})_{(t)} \subseteq \mathfrak{X}^t$, следовательно
\begin{eqnarray*}
\bar{\partial}_z(v)\in \mathfrak{X}^l\mod{(R+N_{1l})_U},~z\in {\bf N}.
\end{eqnarray*}\\
Так как $H\cap (R+N_{1l})=H\cap N_{1l}=(H\cap N_{11})_{(l)} $, то $\bar{\partial}_z(v)\in \mathfrak{X}^l,~z\in {\bf N}$.
Это означает, что $v\in \mathfrak{X}^{l+1}$, следовательно $v\in (H\cap N_{11})_{(l+1)}$.
Теперь соображения индукции заканчивают доказательство.
\end{proof}

\noindent Из предложений \ref{tm4}, \ref{tm2}, \ref{tm5}, вытекает
\begin{theorem}
Пусть $F$ --- свободная алгебра Ли с базой $y_1,\ldots,y_n$ $(n>2)$,
$H$ --- подалгебра, порожденная в алгебре $F$ элементами $y_1,\ldots,y_{n-1}$, $N$ --- эндоморфно допустимый идеал алгебры $F$,
\begin{eqnarray}\label{end}
N=N_{11} \geqslant \ldots \geqslant N_{1,m_1+1}=N_{21} \geqslant \ldots \geqslant N_{s,m_s+1},
\end{eqnarray}
где $N_{kl}$ --- l-я степень алгебры $N_{k1}$.\\
Пусть, далее, $r\in N_{1i}\backslash N_{1,i+1}\,(i\leqslant m_1)$, $R = \mbox{\rm ид}_F (r)$.\\
Если (и только если) $r\not\in H+N_{1,i+1}$, то $H\cap (R+N_{kl})=H\cap N_{kl}$, где $N_{kl}$ --- произвольный член ряда {\rm (\ref{end})}.
\end{theorem}
\begin{corollary} \cite{Tl}
Пусть $F$ --- свободная алгебра Ли с базой $y_1,\ldots,y_n$ $(n>2)$,
$H$ --- подалгебра, порожденная в алгебре $F$ элементами $y_1,\ldots,y_{n-1}$,
\begin{eqnarray}\label{end2}
F=F_{11} \geqslant \ldots \geqslant F_{1,m_1+1}=F_{21} \geqslant \ldots \geqslant F_{s,m_s+1},
\end{eqnarray}
где $F_{kl}$ --- l-я степень алгебры $F_{k1}$.\\
Пусть, далее, $r\in F_{ij}\backslash F_{i,j+1}\,(j\leqslant m_i)$, $R = \mbox{\rm ид}_F (r)$.\\
Если (и только если) $r\not\in H+F_{i,j+1}$, то $H\cap (R+F_{kl})=H\cap F_{kl}$, где $F_{kl}$ --- произвольный член ряда {\rm (\ref{end2})}.
\end{corollary}

\end{document}